%% file: main.tex
\documentclass[11pt,a4paper]{article} %font size and paper size
\usepackage[left=2.5cm, right=2.5cm, top=2.5cm]{geometry} % Sets the page margins
\usepackage[utf8]{inputenc}
\usepackage{graphicx}
\usepackage[square,numbers]{natbib}

\usepackage{float}
\usepackage{amsmath,amsfonts,amsthm,bm,mathtools, amssymb}  %maths packages
\usepackage[UKenglish]{babel}% http://ctan.org/pkg/babel
\usepackage[UKenglish]{isodate}% http://ctan.org/pkg/isodate
\usepackage[nottoc]{tocbibind} %adds references to contents page
\usepackage[hyphens]{url} %adds url format for writing urls with \url{}
\usepackage{enumerate}% http://ctan.org/pkg/enumerate
\usepackage{booktabs}

\usepackage[toc,page,title,titletoc]{appendix}

\usepackage{pgfplots}
\pgfplotsset{compat=newest}

\usepackage{comment}

\usepackage[mathsymb=mathbb]{exmath}
\hidedetails

\usepackage[unicode]{hyperref} %adds links to sections and references
\usepackage[capitalize,nosort,nameinlink]{cleveref}

\DeclareSymbolFont{bbold}{U}{bbold}{m}{n}
\DeclareSymbolFontAlphabet{\mathbbold}{bbold}

\numberwithin{equation}{section} %sets equations to be numbered by section i.e 3.1
\numberwithin{figure}{section} %same for figures
\numberwithin{table}{section} %same for tables
\theoremstyle{definition}

\newtheorem{definition}{Definition}[section]

\newtheorem{theorem}[definition]{Theorem}
\crefname{theorem}{Theorem}{Theorems}
\AddToHook{env/theorem/begin}{\crefalias{definition}{theorem}}

\newtheorem{proposition}[definition]{Proposition}
\crefname{proposition}{Proposition}{Propositions}
\AddToHook{env/proposition/begin}{\crefalias{definition}{proposition}}

\newtheorem{lemma}[definition]{Lemma}
\crefname{lemma}{Lemma}{Lemmas}
\AddToHook{env/lemma/begin}{\crefalias{definition}{lemma}}

\newtheorem{remark}[definition]{Remark}
\crefname{remark}{Remark}{Remarks}
\AddToHook{env/remark/begin}{\crefalias{definition}{remark}}

\newtheorem{corollary}[definition]{Corollary}
\crefname{corollary}{Corollary}{Corollary}
\AddToHook{env/corollary/begin}{\crefalias{definition}{corollary}}

\newtheorem{example}[definition]{Example}
\crefname{example}{Example}{Examples}
\AddToHook{env/example/begin}{\crefalias{definition}{example}}

\crefname{condition}{Condition}{Conditions}
\AddToHook{env/condition/begin}{\crefalias{definition}{condition}}

\crefname{notation}{Notation}{Notations}
\AddToHook{env/notation/begin}{\crefalias{definition}{notation}}
\newtheorem{assumption}[definition]{Assumption}
\crefname{assumption}{Assumption}{Assumptions}
\AddToHook{env/assumption/begin}{\crefalias{definition}{assumption}}

\crefname{property}{Property}{Propertys}
\AddToHook{env/property/begin}{\crefalias{definition}{property}}

\makeatletter
\let\c@table\c@figure % for (1)
\let\ftype@table\ftype@figure % for (2)
\makeatother

\makeatletter
\newcommand{\eqnum}{\hfill\refstepcounter{equation}\textup{\tagform@{\theequation}}}
\makeatother

\usepackage{xpatch}
\makeatletter
\xpatchcmd{\@thm}{\thm@headpunct{.}}{\thm@headpunct{}}{}{}
\makeatother

\newcommand*{\myin}{$\,$\mathord{\in}$\,$}

\crefname{equation}{}{}

\usepackage{listings}
\usepackage[font=small,skip=0.1pt,belowskip=-10pt]{caption}

\AtBeginEnvironment{appendices}{%
 \titleformat{\section}{\bfseries\Large}{\appendixname~\thesection:}{0.5em}{}%
 \titleformat{\subsection}{\bfseries\large}{\thesubsection}{0.5em}{}%
\counterwithin{equation}{section}
}

\usepackage{hhline}

\usepackage{enumitem}

\usepackage{algorithm}
\usepackage[noend]{algpseudocode}

\usepackage{subcaption}

\def\ind#1{\mathbbold{1}_{#1}}

\newcommand\numberthis{\addtocounter{equation}{1}\tag{\theequation}}

\title{
{An Adaptive Sampling Algorithm for Level-set Approximation}
}
\author{Matteo Croci, Abdul-Lateef Haji-Ali, and Ian C. J. Powell}
\date{}

\allowdisplaybreaks

\begin{document}

\maketitle

%\tableofcontents
%\medskip
%\medskip

\begin{abstract}
We propose a new numerical scheme for approximating level-sets of Lipschitz multivariate functions which is robust to stochastic noise.
The algorithm's main feature is an adaptive grid-based stochastic approximation strategy which automatically refines the approximation over regions close to the level set.
This strategy combines a local function approximation method with a noise reduction scheme and produces $\varepsilon$-accurate approximations with an expected cost complexity reduction of $\varepsilon^{-\left(\frac{p+1}{\alpha p}\right)}$ compared to a non-adaptive scheme, where $\alpha$ is the convergence rate of the function approximation method and we assume that the noise can be controlled in $L^p$.
We provide numerical experiments in support of our theoretical findings. These include 2- and 3-dimensional functions with a complex level set structure, as well as a failure region estimation problem described by a hyperelasticity partial differential equation with random field coefficients.
\end{abstract}

%\newpage
%\tableofcontents
%\newpage

\input{intro}
\input{interpolation}
\input{algo}
\input{analysis_work}
\input{analysis_error}

\input{criteria}

\input{results}

%\include{curr}
%\include{Notes}
%\newpage
\bibliography{ref}
%to add references, put them in the REF.bib file

%%Appendix

%\newpage
%\begin{appendices}

%\end{appendices}

\end{document}

%% file: intro.tex
\section{Introduction}

This paper proposes an efficient, grid-based adaptive sampling algorithm for approximating zero level-sets of the form
\begin{equation} \label{eq:level_set_form}
    \mathcal{L}_0 := \{{x} \in \bar{D} : f(x) = 0\}
\end{equation}
where $D\subset\joinrel\subset\mathbb{R}^d$ is a $d$-dimensional domain with compact closure and a
sufficiently smooth boundary, and the Lipschitz function $f:D \rightarrow \mathbb{R}$ can either
be directly evaluated pointwise or can be estimated pointwise by a (random) approximation. 
More concretely, we assume that for a fixed $p\geq1$ there exist (random) approximations of $f$ for which the maximum $L^p$ error of pointwise evaluations can be controlled for some known cost. 
The adaptive sampling algorithm presented here is most effective when the cost of pointwise evaluations/estimations at the desired error tolerance is high.

Level-set estimation is central in a wide range of practical applications, including solar cell manufacturing \cite{inatsu2024active}, environmental risk estimation \cite{gotovos2013active}, the approximation of failure domains in uncertainty quantification \citep{elfverson2016multilevel,chaudhuri2021mfegra}, image segmentation and object extraction in computer vision \citep{osher2003level,willett2007minimax}, and the approximation of moving interfaces in computational mathematics and engineering \citep{min2007second,xu2004adaptive}.
Failure domains are often characterised as the region in which a given function exceeds some failure threshold \citep{elfverson2016multilevel,PeherstorferMultifidelity2018} and thus their approximation can be thought of as level-set estimation problems. This means that level-set estimation is crucial in applications such as engineering and manufacturing design where failure regions are of interest.
Another practical use of level-set estimation arises in scientific and engineering problems that can be formulated as moving free boundary models \citep{min2007second}, where at each timestep in a numerical scheme the location of the boundary must be (approximately) determined. Although explicit boundary estimation methods are available, level-set methods are widely used in practice. These approaches implicitly define the boundary as the level set of an evolving function \citep{xu2004adaptive}. Level-set techniques have found successful applications in diverse areas, including fluid mechanics \citep{sussman1999adaptive}, medical imaging \citep{osher2003level}, and topology optimisation \citep{li2021full}, highlighting their broad applicability.

When pointwise evaluations are expensive or noisy, level-set estimation typically becomes computationally prohibitive and thus adaptive techniques in both space and probability are required to minimize costs and make the estimation feasible.
The main contributions of this paper are as follows:
\begin{itemize}
    \item a new adaptive algorithm for level-set estimation which accounts for noise,
    \item a complete complexity analysis of the adaptive level-set estimation algorithm which covers both deterministic and stochastic (noise-free) cases,
    \item comprehensive numerical results verifying that these theoretical complexity rates hold in practice.
\end{itemize}

The methodology proposed here is perhaps closest in spirit to the work presented in \citep{min2007second}. Indeed, the authors of \citep{min2007second} propose a level-set approximation method for noise-free functions based on adaptive spatial quadtree/octree interpolation. Our work extends well beyond \citep{min2007second} since it accounts for noisy evaluations, adaptive statistical approximation, general grids and cell shapes, and the use of any function approximation scheme.

The structure of the rest of this paper is as follows. \Cref{section:setup} outlines our assumptions and provides an overview of our algorithm. \Cref{section:complexity_analysis} provides an extensive complexity analysis. \Cref{section:numerics} provides numerical results corroborating our complexity analysis. \Cref{section:conclusion} gives concluding remarks.

\section{Problem Setup and Algorithm Overview} \label{section:setup}

In this section we outline the general setup of our adaptive sampling algorithm for level-set estimation. The bulk of this section focuses on stating the assumptions utilised in the complexity analysis in \cref{section:complexity_analysis}, whilst \cref{subsection:adaptive_algorithm} outlines the form of the adaptive sampling algorithm itself.

We first introduce a family of (random) approximations of the function $f$, $\{ \Tilde{f}_\ell \}_{\ell \in \mathbb{N}}$ defined over a probability space \((\Omega, \mathcal F, \mathbb P)\). Here \(\ell \in \mathbb{N}\) denotes the ``level'' of approximation which corresponds to higher accuracy and cost as \(\ell\) increases. We assume that everywhere in $\bar{D}$ the point
evaluations of the $\{ \Tilde{f}_\ell \}_{\ell \in \mathbb{N}}$ are well-defined and have finite $p$-th moment (in probability).
In order to formalize this requirement, we introduce the Banach space of
everywhere bounded deterministic real-valued functions $u:\bar{D} \rightarrow \mathbb{R}$, which
we denote by $\mathcal{B}(\bar{D},\mathbb{R})$, equipped with the maximum norm
\[
    \lVert u \rVert_{\mathcal{B}(\bar{D},\mathbb{R})} =
    \max_{x\in\bar{D}}|u(x)|.
\]
The space $\mathcal{B}(\bar{D},\mathbb{R})$ is essentially a weaker space than
$C^0(\bar{D})$, the space of continuous functions over $\bar{D}$. We introduce
it to model the fact that we need point evaluations of $\Tilde{f}_\ell$ to be
well-defined, but we cannot assume continuity since the noise in
$\Tilde{f}_\ell$ need not be continuous, for example when different point evaluations arise from independent Monte Carlo simulations.

Since, for each $\ell \in \mathbb{N}$, the approximation $\Tilde{f}_\ell:\bar{D}\times\Omega \rightarrow \mathbb{R}$ is
stochastic, we also require the stochastic version of the above space, namely the
Banach space $\mathcal{B}(\bar{D}, L^p(\Omega))$ of real-valued random fields with everywhere
(in space) bounded $p$-th moment (in probability). We equip $\mathcal{B}(\bar{D},
L^p(\Omega))$ with the norm
\begin{align*}
    \lVert u \rVert_{\mathcal{B}(\bar{D},L^p(\Omega))} =
    \max_{x\in\bar{D}}\mathbb{E}[|u(x)|^p]^{1/p}.
\end{align*}
To fix ideas, note that for $p=2$ this is the space of random fields with
everywhere defined finite variance. The above norm does not require that the $p$-th
moment of $\Tilde{f}_\ell$ varies continuously in space in order to be bounded, and thus models typical situations such as when each point evaluation arises from independent Monte Carlo estimators with possibly different statistical error tolerances.

We can now state our first assumption on $\Tilde{f}_\ell$:
\begin{assumption} \label{assumption:beta}
    For a fixed  $p \geq 1$, assume that $f$ is Lipschitz continuous over $\bar{D}$
    (which implies that $f\in\mathcal{B}(\bar{D},\mathbb{R})$) and that, for each level $\ell \in \mathbb{N}$, we have
    $\Tilde{f}_\ell\in\mathcal{B}(\bar{D}, L^p(\Omega))$ satisfying the bound
    \begin{equation}
            \lVert
                    f({x}) - \Tilde{f}_\ell({x})
                \rVert_{\mathcal{B}(\bar{D},L^p(\Omega))}
        \leq
            \sigma M_\ell^{-\beta}
    \end{equation}
    for some fixed, known rate $\beta > 0$ and (unknown) constant $\sigma \geq 0$, where $M_\ell$ is the cost of a pointwise evaluation of $\Tilde{f}_\ell$.
\end{assumption}

\begin{remark}
    The particular case of $\sigma = 0$ or $\beta \rightarrow \infty$ corresponds to the deterministic case wherein, for each point $x \in \bar{D}$, the function $f(x)$ can be directly evaluated at cost $\mathcal{O}(1)$. In this case the $M_\ell$'s are just constants and thus the subscripts can be safely omitted.
\end{remark}

This general structure occurs frequently in stochastic approximation problems. A typical case with noisy point-values arises when point evaluations are provided via stochastic approximation routines based on standard Monte Carlo methods, which we illustrate with the following simple example.

\begin{example}
    Suppose that, for each point $x \in \bar{D}$, pointwise evaluations are the result of a Monte Carlo approximation of the form
    \begin{equation*}
        \Tilde{f}_\ell({x}) = \frac{1}{M_\ell} \sum_{m = 1}^{M_\ell} g(x, Y^{(m)})
    \end{equation*}
    for some set of (finite variance) i.i.d. samples $\{g(x, Y^{(m)})\}_{m = 1}^{M_\ell}$ with mean $f(x)$ and a given hierarchy $\{M_\ell\}_\ell \subset \mathbb{N}$.
    Then, letting the cost of each sample be $\mathcal{O}(1)$, the cost of obtaining a pointwise evaluation is directly proportional the number of samples used, $M_\ell$. In this case \cref{assumption:beta} typically holds pointwise for $p=2$ and $\beta = 1/2$ \citep{kalos2009monte}.
\end{example}

We also make the following assumption on the behaviour of the function $f$ in a neighbourhood of the target level-set $\mathcal{L}_0$. This assumption bounds the size of sets obtained from inflating the target level-set by some parameter and is utilised throughout the complexity analysis in \cref{section:complexity_analysis}.

\begin{assumption} \label{assumption:f_int_bound}
    There exist some $\delta, \rho_0 > 0$ such that for all $0 < a < \delta$ we have
    \begin{equation*}
        \mu(x\in\bar{D} : |f(x)|\leq a) \leq \rho_0 a
    \end{equation*}
    where $\mu$ is the $d$-dimensional Lebesgue measure.
\end{assumption}

\begin{remark}
    Since we assume that $f$ is Lipschitz, \cref{assumption:f_int_bound} automatically follows from an assumption on the level set $\mathcal{L}_0$ being a finite union of $d-1$ dimensional rectifiable sets, for example a finite union of $d-1$ dimensional submanifolds of class $1$. It is reasonable to expect \cref{assumption:f_int_bound} to hold in practice since, by Theorem $3.2.15$ of \cite{federer2014geometric}, almost all level sets of Lipschitz functions mapping from $\bar{D}$ to $ \mathbb{R}$ are countable unions of $d-1$ dimensional rectifiable sets.
\end{remark}

%% file: interpolation.tex
\subsection{Local Approximation}

Our adaptive sampling scheme relies on a cell-based refinement structure with local, cell-wise approximations where sample complexity (and thus approximation accuracy) chosen relative to the size of the cell. To make this set-up clear we introduce the following notation.

\begin{definition}
    Let $N$ be the (fixed) number of points in space to be used for each local approximation.
    At a given approximation level $\ell \myin \mathbb{N}$, define:
    \begin{itemize}[leftmargin=1.5cm]
        \item $h_\ell$, the level-dependant, geometrically decreasing (e.g. $h_\ell = 2^{-\ell}$) cell-size;
        \item $U_\ell$, a uniform tessellation of $\bar{D}$ into a cells of size $h_\ell$;
    \end{itemize}
    Further, for each cell $\square \in U_\ell$ define:
    \begin{itemize}[leftmargin=1.5cm]
        \item $P_\ell^{\square}:\mathcal{B}({\square},
            \mathbb{R})\rightarrow\mathbb{R}^{N}$, a (deterministic and linear) point evaluation operator
            which evaluates a function at $N$ distinct point-values within the cell;
        \item $T_\ell^{\square}: \mathbb{R}^{N} \rightarrow C^{0}(\bar{\square})$,
            an approximation operator which constructs a (cell-wise continuous)
            function approximation from a length-$N$ vector of function point
            values;
        \item $I_\ell^{\square}: \mathcal{B}({\square},
            \mathbb{R}) \rightarrow  C^{0}(\bar{\square})$, an approximation operator given by $I_\ell^{\square} = T_\ell^{\square} P_\ell^{\square}$.
    \end{itemize}
\end{definition}

We require the operators $\{T_\ell^{\square}\}_{\ell \in \mathbb{N}, \square \in U_\ell}$ to be bounded as well as Lipschitz-continuous with respect to the $L^p(\square)$ norm:
\begin{assumption}
    \label{ass:continuity_of_T}
    For any $\ell \myin \mathbb{N}$ and over any cell $\square \in U_\ell$, the approximation
    operators $T_\ell^{\square}$ are continuous and bounded. Additionally, they are Lipschitz continuous with respect to the $L^p(\square)$ norm, i.e., for all $\ell$ there exist $K^\square_{\ell}>0$ satisfying
    \begin{align*}
        \lVert T_\ell^\square a - T_\ell^\square b \rVert_{L^p({\square})} \leq K^\square_\ell\lVert a - b \rVert_{p},\quad \forall a,b\in \mathbb{R}^N.
    \end{align*}
    Furthermore, the Lipschitz constants $K_\ell^\square$ satisfy
    \begin{align}
        \label{eq:sum_of_operator_norms_bdd}
        \left(\sum_{\square \in U_\ell}
        (K_\ell^\square)^p\right)^{1/p} \leq C_N
    \end{align}
    for some constant $C_N>0$ independent of the level $\ell$.
\end{assumption}
Note that \cref{eq:sum_of_operator_norms_bdd} amounts to requiring that the global approximation operator built by gathering all cell contributions is also Lipschitz-continuous with a mesh-independent Lipschitz constant.

The operators $P^\square_\ell$ are also bounded and
continuous, and can be extended to continuous bounded operators from
$\mathcal{B}(\bar{D}, L^p(\Omega))$ to $L^{p}(\Omega, \mathbb{R}^{N})$, as stated by the
following lemma.

\begin{lemma}
    \label{thm:continuity_of_P}
    For any $\ell \myin \mathbb{N}$ and over any cell $\square \in U_\ell$, the operators
    $P^\square_\ell$ are bounded and continuous and can be extended to continuous and
    bounded operators from $\mathcal{B}(\square, L^p(\Omega))$ to
    $L^{p}(\Omega, \mathbb{R}^{N})$. Indeed, the following bounds hold
    \begin{align}
        \label{eq:continuity_of_P}
        \begin{array}{cl}
            \lVert P^\square_\ell u \rVert_p \leq c_{N,p} \lVert u
        \rVert_{\mathcal{B}(\square,\mathbb{R})} & \forall u\in
        \mathcal{B}(\square,\mathbb{R}),\\
        \mathbb{E}\left[\lVert P^\square_\ell u \rVert_p^p\right]^{1/p} \leq
        c_{N,p} \lVert u
        \rVert_{\mathcal{B}(\square,L^p(\Omega))} & \forall u\in
        \mathcal{B}(\square,L^p(\Omega)),
        \end{array}
    \end{align}
    where $c_{N,p} = (N(N+1)/2)^{1/p}$. Likewise, the operators $T_\ell^\square$
    can be extended to bounded operators from $L^p(\Omega,\mathbb{R}^N)$ to $L^p(\Omega,
    C^0(\bar{\square}))$ with the same operator norm. i.e.,
    \begin{align}
        \label{eq:extension_of_T_opnorm}
        \|T_\ell^{\square}\|_{\mathcal{L}(L^p(\Omega, \mathbb{R}^{N}),{L^p(\Omega,
        C^0(\bar{\square}))})}=\|T_\ell^{\square}\|_{\mathcal{L}(\mathbb{R}^{N},{C^0(\bar{\square})})}.
    \end{align}
    The extended $T_\ell^\square$ are still Lipschitz-continuous with respect to the $L^p(\square)$ norm with the same Lipschitz constants, i.e.,
    \begin{align}
        \lVert T_\ell^\square a - T_\ell^\square b \rVert_{L^p(\Omega, L^p({\square}))} \leq K_\ell^\square\lVert a - b \rVert_{L^p(\Omega, \mathbb{R}^N)},\quad \forall a,b\in L^p(\Omega,\mathbb{R}^N).
    \end{align}
\end{lemma}
\begin{proof}
    Let $\{x_i\}_{i=1}^N\subset \bar{\square}$ be the evaluation points used by
    the point evaluation operator. For $u\in \mathcal{B}(\square,\mathbb{R})$, we have that
    \begin{align*}
        \lVert P^\square_\ell u \rVert_p^p = \sum_{i=1}^{N}|u(x_i)|^p \leq
        \left(\sum_{i=1}^N 1\right) \lVert u
        \rVert_{\mathcal{B}(\square,\mathbb{R})}^p
        \leq c_{N,p}^p\lVert u
        \rVert_{\mathcal{B}(\square,\mathbb{R})}^p,
    \end{align*}
    which is the first bound in \eqref{eq:continuity_of_P}.
    Similarly for $u\in \mathcal{B}(\square,L^p(\Omega))$, we have that point
    values of $u$ are random variables in $L^p(\Omega)$ and thus
    \begin{align*}
        \mathbb{E}\left[\lVert P^\square_\ell u \rVert_p^p\right] =
        \sum_{i=1}^N\mathbb{E}\left[|u(x_i,\cdot)|^p\right] \leq c_{N,p}^p\lVert u
        \rVert_{\mathcal{B}(\square,L^p(\Omega))},
    \end{align*}
    which is the second bound in \eqref{eq:continuity_of_P}. To establish
    \eqref{eq:extension_of_T_opnorm},
    note that, for all random vectors $v\in L^{p}(\Omega, \mathbb{R}^{N})$ it holds that
    \begin{align*}
            \lVert Tv \rVert_{L^p(\Omega,
            C^0(\bar{\square}))}^p
        &=
            \mathbb{E}\left[\max_{x\in\bar{\square}}|Tv|^p(x,\omega)\right]
        \\&\leq
            \mathbb{E}\left[\|T_\ell^{\square}\|^p_{\mathcal{L}(\mathbb{R}^{N},{C^0(\bar{\square})})}\lVert
            v\rVert^p_p\right]
        \\&=
            \|T_\ell^{\square}\|^p_{\mathcal{L}(\mathbb{R}^{N},{C^0(\bar{\square})})}\lVert
            v \rVert_{L^{p}(\Omega, \mathbb{R}^{N})}^{p}
    \end{align*}
    which implies that $\|T_\ell^{\square}\|_{\mathcal{L}(L^p(\Omega, \mathbb{R}^{N}),{L^p(\Omega,
    C^0(\bar{\square}))})}\leq\|T_\ell^{\square}\|_{\mathcal{L}(\mathbb{R}^{N},{C^0(\bar{\square})})}$.
    To prove that the reverse inequality also holds, note that for all $u\in
    C^0(\bar{\square}),\ v\in\mathbb{R}^N$ we have $\lVert u \rVert_{L^p(\Omega,
    C^0(\bar{\square})} = \lVert u \rVert_{C^0(\bar{\square})}$ and $\lVert v
    \rVert_{L^{p}(\Omega, \mathbb{R}^{N})}=\lVert v \rVert_p$, and thus, for all
    $v\in\mathbb{R}^N$ we have
    \begin{align*}
            \frac{\lVert T_\ell^{\square}v \rVert_{C^0(\bar{\square})}}{\lVert v \rVert_{\mathbb{R}^N}}
        =
            \frac{\lVert T_\ell^{\square}v \rVert_{L^p(\Omega, C^0(\bar{\square}))}}{\lVert v
            \rVert_{L^{p}(\Omega, \mathbb{R}^{N})}} \leq \lVert T_\ell^{\square} \rVert_{\mathcal{L}(L^p(\Omega, \mathbb{R}^{N}),{L^p(\Omega,C^0(\bar{\square}))})}.
    \end{align*}
    Taking the supremum over all $v\in\mathbb{R}^n$ we thus obtain
    \begin{align*}
       \|T_\ell^{\square}\|_{\mathcal{L}(\mathbb{R}^{N},{C^0(\bar{\square})})}
       = \sup_{v\in\mathbb{R}^N}\frac{\lVert T_\ell^{\square}v \rVert_{C^0(\bar{\square})}}{\lVert v
    \rVert_{\mathbb{R}^N}} \leq
        \|T_\ell^{\square}\|_{\mathcal{L}(L^p(\Omega, \mathbb{R}^{N}),{L^p(\Omega,
        C^0(\bar{\square}))})}
    \end{align*}
    which, together with the other side of the inequality we just derived,
    implies that the two operator norms are the same. To establish Lipschitz continuity it is sufficient to use the fact that all $a,b\in L^p(\Omega,\mathbb{R}^N)$ are a.s. in $\mathbb{R}^N$ and thus
    \begin{align*}
        \lVert T_\ell^\square a - T_\ell^\square b \rVert_{L^p({\square})} \leq K_\ell^\square\lVert a - b \rVert_{p}\quad a.s.,\quad\forall a,b\in L^p(\Omega, \mathbb{R}^N).
    \end{align*}
    Taking the $p$-th moment of both sides concludes the proof.
\end{proof}

Due to \cref{ass:continuity_of_T} and to
\cref{thm:continuity_of_P} we obtain that the operators $I_\ell^\square$ are
also bounded since they are composition of bounded operators:
\begin{corollary}
    \label{coroll:continuity_of_I}
    For all $\ell \in \mathbb{N}$ and over any cell $\square \in U_\ell$, the approximation operators $I_\ell^\square$
    are well-defined bounded operators between
    $\mathcal{B}(\bar{\square}, \mathbb{R})$ and $ C^0(\bar{\square})$, and can
    be extended to be bounded operators between
    $\mathcal{B}(\bar{\square}, L^p(\Omega))$ and $L^p(\Omega,
    C^0(\bar{\square}))$.
\end{corollary}

Before proceeding, we also require an assumption about the convergence rate of
the local approximation operators $I_\ell^\square$ when applied to the
deterministic true function $f$:

\begin{assumption} \label{assumption:interpolant_rate_alpha}
    For any $\ell \myin \mathbb{N}$ a level-$\ell$ uniform
    refinement $U_\ell$ satisfies
            \begin{align*}
                    \left(
                        \sum_{\square \in U_\ell}
                        \left\|
                            f - I^{\square}_\ell f
                        \right\|_{L^{p}(\square)}^p
                    \right)^{1/p}
                \leq
                    c_{I} h_{\ell}^{\alpha}
            \end{align*}
        for a known rate $\alpha>\left(\frac{p+1}{p}\right)$,
            where $c_I >0$ is an (unknown) constant.
\end{assumption}

This assumption provides us with a canonical rate, $\alpha$, for the global error
of our cell-wise approximations of a level-$\ell$ uniform refinement $U_\ell$ in the
deterministic case where we have access to the exact function $f$ for which we
want to approximate the level set $\mathcal{L}_0$. Since in practice we assume that we
instead only have access to point-evaluations of the (random) functions
$\Tilde{f}_\ell$ rather than the exact function $f$, we must instead rely
on a (random) functional approximation of the form given in the following definition.

\begin{definition}
    For each cell $\square \in U_\ell$ define the (random) functional approximation
    \begin{align*}
        \hat{f}^{\square}_\ell := I^{\square}_\ell \Tilde{f}_\ell.
    \end{align*}
    Note that $\hat{f}^{\square}_\ell \equiv I^{\square}_\ell f$ in the deterministic setting (in which $\beta \rightarrow \infty$).
\end{definition}

The following result ensures that we can control the expected $L^p$ error of the global (random) estimate by refining our grid and increasing the cost of point evaluations (e.g. increasing the number of samples used in Monte Carlo estimations) by a fixed amount.

\begin{proposition} \label{proposition:epsilon_ell-alpha_rate_sufficent_condition}
    Let \cref{assumption:interpolant_rate_alpha,assumption:beta,ass:continuity_of_T} hold for fixed values of $\alpha> \left(\frac{p+1}{p}\right)$, $\beta>0$, and $C_N >0$ respectively.
    Then for $M_\ell = M_0 h_\ell^{-\alpha/\beta}$ and any $\ell \myin \mathbb{N}$, a level-$\ell$ uniform refinement of $\bar{D}$ into a collection of uniform cells, $U_\ell$, each of size $h_\ell$, satisfies
    \begin{align} %\label{eq_lp_fhat}
            \left(
                \sum_{\square \in U_\ell}
                    \left\|
                        f - \hat{f}^{\square}_\ell
                    \right\|_{L^{p}(\Omega, L^p(\square))}^p
            \right)^{1/p}
        &\leq
            \hat{c}_I h_\ell^{\alpha}
    \end{align}
    for some constant $\hat{c}_I >0$ independent of $\ell$.
\end{proposition}
\begin{proof}
    Writing $\hat{f}_{\ell}^{\square} = I^{\square}_\ell \Tilde{f}_{\ell}$, we find that
    \begin{align*}
            \sum_{\square \in U_\ell}
                    \left\|
                        f - \hat{f}^{\square}_\ell
                    \right\|_{L^{p}(\Omega, L^p(\square))}^p
        &\leq
            2^{p-1}\sum_{\square \in U_\ell}
            \left(
                    \left\|
                        f - I^{\square}_\ell f
                    \right\|_{L^{p}(\square)}^p
                +
                    \left\|
                        I^{\square}_\ell f - I^{\square}_\ell \Tilde{f}_{\ell}
                        \right\|_{L^{p}(\Omega, L^p({\square}))}^p
            \right)
        \\&\leq
            c 2^{p-1} h_\ell^{\alpha p}
            +
            2^{p-1}\sum_{\square \in U_\ell}
            (L_\ell^\square c_{N,p})^p\lVert
        f - \tilde{f}_\ell \rVert_{B(\bar{\square}, L^p(\Omega))}^p
        \\&\leq
            c2^{p-1} h_\ell^{\alpha p}
            +
            2^{p-1}(C_Nc_{N,p})^p \lVert
        f - \tilde{f}_\ell \rVert_{B(\bar{D}, L^p(\Omega))}^p
        \\&\leq
            c 2^{p-1}h_\ell^{\alpha p}  + 2^{p-1}(\sigma C_Nc_{N,p})^p M_\ell^{-\beta p}. \numberthis \label{eq:alpha_beta_constants}
    \end{align*}
    In the above, we applied the triangle inequality in the first step. We then used \cref{assumption:interpolant_rate_alpha} to bound the first term and the
    Lipschitz-continuity and boundedness properties of $T_\ell^\square$ and
    $P_\ell^\square$ (cf. \cref{thm:continuity_of_P} and
    \cref{coroll:continuity_of_I}) to bound the second term. Finally, we used
    \cref{ass:continuity_of_T} (and implicitly \cref{thm:continuity_of_P}
    again) to bound the sum of the Lipschitz constants of $T_\ell^\square$ and
    \cref{assumption:beta} to bound the error between $f$ and
    $\tilde{f}_\ell$. We conclude the proof by selecting $M_\ell = M_0
    h_\ell^{-\alpha/\beta}$ which ensures that both terms are
    $\mathcal{O}(h_\ell^{\alpha p})$ as in the thesis.
\end{proof}

%% file: algo.tex
\subsection{The Adaptive Algorithm} \label{subsection:adaptive_algorithm}

In this section we outline the general form of the adaptive refinement algorithm. Leveraging the cell-based structure outlined in the previous section, our algorithm treats each cell independently, allowing for local refinement decisions. 
By default, this results in a piecewise-continuous approximation over a non-graded mesh. We note that continuous approximations are also possible and fit within our theory. 

For our local refinement decisions we require access to both a local decision variable and a refinement criterion to check the decision variable against. In particular, for a given cell $\square$, we make use of the local, cell-specific (random) decision variable
\begin{equation} \label{eq:delta_hat}
    \hat{\delta}^{\square}_{\ell} := \frac{\inf_{x \in \square}  |\hat{f}^{\square}_{\ell}(x)|}{ h_{\ell}^{\alpha}}
\end{equation}
which models the sample- and cell-specific uncertainty in the presence of the target zero level-set $\mathcal{L}_0$ in the given cell, inspired by the sample specific uncertainty in \cite{haji2022adaptive}. The numerator gives a notion of the distance between the given cell $\square$ and the target zero level-set $\mathcal{L}_0$, whilst the denominator is a measure of the accuracy of the estimate $\hat{f}^{\square}_{\ell}$, in particular an upper-bound of the $L^p$ error via \cref{proposition:epsilon_ell-alpha_rate_sufficent_condition}.

The idea behind the decision variable $\hat{\delta}^{\square}_{\ell}$ is that it should be small (and we should refine the cell $\square$) when either the numerator is small (indicating that the local approximation is near its level-set), or the denominator is large (indicating that the approximation is inaccurate). In this way this decision variable captures some notion of our certainty that the target level-set $\mathcal{L}_0$ is contained within the cell in question $\square$.

\begin{remark}
    The decision variable presented in \eqref{eq:delta_hat} is based on a priori error estimates.
    One can also substitute a posteriori error estimates in place of the upper-bound in the denominator of \eqref{eq:delta_hat} for sharper bounds and better constants in the computational complexity. Indeed, by using an error estimate smaller than $h_{\ell}^\alpha$, one obtains a larger $\hat{\delta}^{\square}_{\ell}$, leading to less refinement and thus less work. 
    The complexity analysis in \cref{section:complexity_analysis} still holds with minor modifications if one utilises a posteriori error estimates in the decision variable. In particular, it is sufficient to add an additional assumption similar to \cref{assumption:beta} but with sample-specific a posteriori error estimates appearing within the LHS expectation.
\end{remark}

The algorithm we propose adaptively refines from a level $\ell_0$ corresponding to a uniform grid, to a target refinement level $L$, which can be selected based on accuracy requirements.
Thus, we require a refinement criteria to check our local, cell-wise (random) decision variable against so that we know which cells should be refined in our adaptive scheme. Since the magnitude of $\hat{\delta}^{\square}_{\ell}$ is expected to change with the accuracy level $\ell$, we select a level-dependant refinement criteria. 
In particular, we denote the sequence of refinement criteria by 
\(
    \{a_{\ell}\}_{\ell=\ell_0}^{{L}-1}
\)
where each $a_{\ell} > 0$. Then, for each cell $\square$ in a given level-$\ell$ tessellation of $\bar{D}$, we compare the refinement criteria with our decision variable and refine cells in which we find that $\hat{\delta}^{\square}_{\ell} \leq a_{\ell}$.

In \cref{section:complexity_analysis} we derive the conditions for which the refinement criteria lead to optimal cost complexity. However, for clarity we provide a specific choice here which satisfies these conditions:
\begin{equation} \label{eq:define_a1}
        a_{\ell}
    = 
        c h_{\ell}^{\alpha_p/R} h_{{L}}^{\alpha_p (R-1)/ R} h_{\ell_{0}}^{-\alpha_p/R} h_{\ell}^{-\alpha}
\end{equation}
for some chosen constant $c>0$, where $p\geq 1$ is as in \cref{assumption:beta,assumption:interpolant_rate_alpha}, $\alpha_p = \alpha\left(\frac{p}{p+1}\right) > 1$, the parameter $R$ satisfying $1 < R < \alpha_p$ controls the strictness of refinement (we find that selecting $R=(1+\alpha_p)/2$ typically works well in practice), and $\ell_{0} = \left\lceil {L} \left(1 - \frac{p}{R(p+1)}\right)\right\rceil$ (or as in \cref{remark:vartheta_general} if error constants are known). This criteria is adapted from \cite{haji2022adaptive} and obtains optimal cost-complexity. Interested readers should refer to \cref{section:complexity_analysis} and in particular \cref{subsection:refinement_criteria} for further details on this choice refinement criteria and its optimality.

\cref{algo:main_algo} below outlines the general form of the adaptive refinement scheme for level-set estimation considered in this paper. It takes a uniformly refined mesh $U_0$, uniformly refines it $\ell_{0}$ times before performing cell-wise adaptive refinement to obtain a new mesh $\mathcal{R}_{{L}}$ and returns the level-set estimate based on evaluations of the cell-wise approximated level-sets of the new mesh. 
The algorithm requires the user to prescribe a final refinement level $L$. In practice, it may often be more convenient to prescribe an error tolerance $\varepsilon$ instead (using the error metric given in \cref{eq:error_metric}). 
It is then sufficient to select $L$ such that
$
    h_{L} = \mathcal{O}\left(\varepsilon^{1/\alpha_p}\right)
$.
For further details see \cref{subsection:selection_base_refinement} and in particular \cref{theorem:final_complexity}.

\begin{algorithm}
\caption{Adaptive refinement scheme for level-set estimation, given a general uniform refinement.}
\label{algo:main_algo}
\begin{algorithmic}
    \Require the uniform grid $U_0$ to be refined;
        the interpolation rate $\alpha$ as in \cref{assumption:interpolant_rate_alpha};
        the sampling rate $\beta$ as in \cref{assumption:beta};
        the dimension of the problem $d$;
        the target maximum refinement level ${L}$;
        the number of point evaluations in each cell $N$;
        the cost of point evaluations at level $0$ $M_0$;
        the size $h_0$ of cells in $U_0$;
        the sequence of refinement criteria $\{a_{\ell}\}_{\ell=\ell_{0}}^{{L}-1}$.
    \State Set $\mathcal{R}_0 = U_0$; \Comment{Begin with the given uniform refinement}
    \State Refine $\mathcal{R}_0$ uniformly $\ell_{0}$ times to obtain $\mathcal{R}_{\ell_{0}}$;
    \For{$\ell \in \{\ell_{0}, \ldots, {L}-1\}$}
        \For{each cell $\square$ in $\mathcal{R}_{\ell}$ of size $h_{\ell}$ } 
        \Comment{Iterate over cells of the current level}
                \State Evaluate $\Tilde{f}_{\ell}$ at the $N$ evaluation points in $\square$;
                    \Comment{e.g. for $M_{\ell} \propto h_{\ell}^{-\alpha/\beta}$ cost each}
                \State Fit the cell-based estimate $\hat{f}^{\square}_{\ell}$ on the sampled values of $\Tilde{f}_{\ell}$;
                \State Compute decision variable $\hat{\delta}^{\square}_{\ell}$;
            \If { $\hat{\delta}^{\square}_{\ell} \leq a_{\ell} $ } 
                \State Refine $\square$ into multiple cells of size $h_{\ell+1}$, and add them to $\mathcal{R}_{\ell+1}$
            \Else
                \State Add $\square$ to $\mathcal{R}_{\ell+1}$;
            \EndIf
        \EndFor
    \EndFor
    \State Return the union of $\{\hat{f}^{\square}_{{L}}\}_{\square \in \mathcal{R}_{{L}} }$ zero level-sets.
    \Comment{The final level-set estimate}
\end{algorithmic}
\end{algorithm}

If the relevant constants are known, the cell-size of the base uniform refinement $h_0$ and the cost of point evaluations in the base refinement $M_0$ should be selected so that the terms in \cref{eq:alpha_beta_constants} are of the same magnitude. These choices ensure that the error is not dominated by either the deterministic or stochastic error, which should provide better constants in the work rates.

\begin{remark}
    Note that increasing the maximum refinement level a posteriori does not require re-running \cref{algo:main_algo} from scratch. Indeed, one can run the algorithm on a refined grid since the algorithm only checks cells of a specified size during iteration, meaning that already refined cells will not be considered for refinement until the corresponding level is reached.
\end{remark}

%% file: analysis_work.tex
\section{Complexity Analysis} \label{section:complexity_analysis}

In this section we derive the computational complexity of \cref{algo:main_algo}. We do this in two steps: \cref{section:work} focuses on bounding the work, whilst \cref{section:error} focuses on bounding the error. \Cref{subsection:refinement_criteria} and \cref{subsection:selection_base_refinement} bring these results together to obtain a final computational complexity result for a certain choice of refinement criteria.

\subsection{Work Analysis} \label{section:work}

At a given level $\ell$, we define the work of a cell $\square \in U_\ell$ to be the cost of constructing the local cell approximation $\hat{f}^\square_\ell$. Under the assumptions of \cref{proposition:epsilon_ell-alpha_rate_sufficent_condition} such an approximation can be constructed for a cost proportional to
\(
    N M_\ell \propto N h_{\ell}^{-\alpha/\beta}
\)
since $\hat{f}^\square_\ell$ can be constructed from $N$ point evaluations of the random function $\Tilde{f}_\ell$, with each evaluation having cost $M_\ell$, which is chosen as in \cref{proposition:epsilon_ell-alpha_rate_sufficent_condition} in order to assure the correct (expected) error rate is obtained.

With this in mind, we define the (random) work of the method on a cell $\square \in U_\ell$ by the recursive formula
\begin{align} \label{equation:def_work_recursive_cell}
            W_{\ell}^{\square}
    &:=
            N M_\ell
        +
            \ind{\hat{\delta}_{\ell}^{\square} < a_\ell}
            \sum_{\square' \in \mathcal{R}(\square)}
                W_{\ell+1}^{\square'}
\end{align}
where $N$ is the number of points used in each cell, $M_\ell$ is the cost of each point evaluation at level $\ell$, and $\mathcal{R}(\square)$ is the collection of cells obtained from refining the cell $\square$. Our work can be written this way by the setup of \cref{algo:main_algo}: the first term is the cost of the level-$\ell$ local approximation on the cell, whilst the second term is the cost of the additional refinements which only occur when $\hat{\delta}_{\ell}^{\square} < a_\ell$.

Under this framework, the work of our adaptive sampling scheme for level set estimation starting from a generic level-$\ell$ uniform refinement $U_\ell$ is given by
\begin{align} \label{equation:def_work_recursive}
        \sum_{\square \in U_\ell}
            W_{\ell}^{\square}
    &:=
        \sum_{\square \in U_\ell} \left(
                N M_\ell
            +
                \ind{\hat{\delta}_{\ell}^{\square} < a_\ell}
                \sum_{\square' \in \mathcal{R}(\square)}
                    W_{\ell+1}^{\square'}
        \right)
\end{align}
which is the total work of the method once completed.

Since the total work of our method is proportional to the number of cells obtained after refinement, we first seek to bound this. We start by considering the simpler scenario in which the underlying function $f$ is known. This bound can then be used to evaluate the performance of our method compared to the best case scenario when one has full knowledge.

\cref{assumption:f_int_bound} allows us to bound the number of cells of a certain size which contain the set
$
            \mathcal{L}_{a}^{+}
        =
            \{
                x \in \bar{D}
                    \,\Big\vert\,
                |f(x)|\leq a
            \}
$
whenever the target function $f$ is Lipschitz on $\bar{D}$. The following lemma gives the details.

\begin{lemma} \label{lemma:num_cells_f<a}
    Let the assumptions of \cref{proposition:epsilon_ell-alpha_rate_sufficent_condition} hold, and assume that $f$ is Lipschitz on $\bar{D}$ with Lipschitz constant $K$.
    Then, under \cref{assumption:f_int_bound} with a given $\delta>0$, a level-$\ell$ uniform refinement $U_\ell$ is such that, for any $0 \leq a < \delta - K 2^{d/2} h_{\ell}$
    \begin{align*}
        \sum_{\square \in U_\ell}
            \sup_{x \in \square}
            \ind{|f(x)|\leq a}
        \leq
            b h_{\ell}^{-(d-1)} + c a h_{\ell}^{-d}
    \end{align*}
    for some constants $b, c >0$ independent of $\ell$.
    \begin{proof}
        We begin by noting that for any choice of $x,y \myin \square$ we find that
        \begin{align*}
                \abs[\big]{ \, \abs{f(x)} - \abs{f(y)} \, }
            &\leq
                \left| f(x) - f(y) \right|
            \leq
                K \norm*{x - y}_{2}
            \leq
                K 2^{d/2} h_{\ell}
        \end{align*}
        where $K$ is the Lipschitz constant of $f$ on $\bar{D}$, $\|\cdot\|_2$ denotes the Euclidean norm, and $2^{d/2} h_{\ell}$ is the diameter of a cell of size $\ell$. Thus, if $\exists \, x \myin \square$ such that $|f(x)| \leq a$, then for any choice of $y \myin \square$ we have that
        \begin{align*}
                |f(y)|
            &\leq
                K 2^{d/2} h_{\ell} + |f(x)|
            \leq
                K 2^{d/2} h_{\ell} + a.
        \end{align*}
        Thus, we find that, whenever $K 2^{d/2} h_{\ell} + a < \delta$, where $\delta$ is as in \cref{assumption:f_int_bound}, we have that $\sup_{x \in \square} \ind{|f(x)|\leq a} \leq \ind{|f(y)|\leq K 2^{d/2} h_{\ell} + a}$ for all $y \in \square$. From this we obtain
        \begin{align*}
                \sum_{\square \in U_\ell}
                    \sup_{x \in \square}
                    \ind{|f(x)|\leq a}
            &=
                h_{\ell}^{-d} \sum_{\square \in U_\ell}
                    \int_{\square}
                    \sup_{x \in \square}
                    \ind{|f(x)|\leq a}
                    \,d\mu(y)
            \\&\leq
                h_{\ell}^{-d} \sum_{\square \in U_\ell}
                    \int_{\square}
                    \ind{|f(y)|\leq K 2^{d/2} h_{\ell} + a}
                    \,d\mu(y)
            \\&\leq
                h_{\ell}^{-d}
                    \int_{\bar{D}}
                    \ind{|f(y)|\leq K 2^{d/2} h_{\ell} + a}
                    \,d\mu(y)
            \\&\leq
                h_{\ell}^{-d}
                    \rho_0 (K 2^{d/2} h_{\ell} + a)
        \end{align*}
        where $h_{\ell}^d$ is the volume of a cell of size $h_{\ell}$, and we used \cref{assumption:f_int_bound} to bound the measure of the set.
    \end{proof}
\end{lemma}

\begin{remark} \label{remark:work_optimal}
    Note that in \cref{lemma:num_cells_f<a} the case where $a=0$ corresponds to a bound on the number of cells (of size $h_{\ell}$) required to cover the target zero level-set $\mathcal{L}_0$. It is significant that one only requires $\mathcal{O}(h_{\ell}^{-(d-1)})$ cells to obtain coverage in this case, suggesting that an optimal work rate for an adaptive method (i.e. one that only refines exactly over $\mathcal{L}_0$) is $\mathcal{O}(N M_{L} h_{{L}}^{-(d-1)}) = \mathcal{O}(N h_{{L}}^{-(\frac{\alpha}{\beta} + d-1)})$. This rate will serve as our benchmark for work optimality.
\end{remark}

In contrast to \cref{lemma:num_cells_f<a}, our method does not assume direct access to $\inf_{x \in \square} f(x)$. Instead, we utilise the estimate $\inf_{x \in \square} \hat{f}_\ell^{\square}(x)$, which motivates the following corollary.

\begin{corollary} \label{corollary:num_cells_fhat_ell}
    Let the assumptions of \cref{proposition:epsilon_ell-alpha_rate_sufficent_condition} hold, and let $f$ be Lipschitz in the standard Euclidean norm  on $\bar{D}$ with Lipschitz constant $K$.
    Then, under \cref{assumption:f_int_bound}, a level-$\ell$ uniform refinement $U_\ell$ is such that, for any $0 \leq a < \delta - K 2^{d/2} h_{\ell}$
    \begin{align}
        \sum_{\square \in U_\ell}
            \mathbb{E} \!\left[
            \sup_{x \in \square}
                \ind{|\hat{f}_{{\ell}}^{\square}(x)|\leq a}
            \right]
        \leq
            c_1 h_{\ell}^{-(d-1)} +
                \left(
                    c_2 a +
                    c_3 h_{\ell}^{\alpha_p}
                \right) h_{\ell}^{-d}
        \end{align}
        for some constants $c_1, c_2, c_3 > 0$ independent of $\ell$ where $\alpha_p = \alpha \left(\frac{p}{p+1}\right)$.
    \begin{proof}
        Similarly to our proof of \cref{lemma:num_cells_f<a},
        we begin by noting that for any choice of $x, y \myin \square$ one can use the triangle inequality to see that 
        \begin{align*}
                \left| |\hat{f}_{{\ell}}^{\square}(x)| - |\hat{f}_{{\ell}}^{\square}(y)| \right|
            &\leq
                \left| \hat{f}_{{\ell}}^{\square}(x) - f(x) \right|
                + \left| \hat{f}_{{\ell}}^{\square}(y) - f(y) \right|
                + \left| f(x) - f(y) \right|
            \\&\leq
                \left| \hat{f}_{{\ell}}^{\square}(x) - f(x) \right|
                + \left| \hat{f}_{{\ell}}^{\square}(y) - f(y) \right|
                + K 2^{d/2} h_{\ell}
        \end{align*}
        almost surely, where $K$ is the Lipschitz constant of $f$ on $\bar{D}$, and $2^{d/2} h_{\ell}$ is the maximum diameter of a cell of size $\ell$. Thus, if $\exists \, x \myin \square$ such that $|\hat{f}_{{\ell}}^{\square}(x)| \leq a$, then, for any choice of $y \myin \square$, we must have that
        \begin{align*}
                |\hat{f}_{{\ell}}^{\square}(y)|
            &\leq
                \left| \hat{f}_{{\ell}}^{\square}(y) - f(y) \right|
                + \left| \hat{f}_{{\ell}}^{\square}(x) - f(x) \right|
                + K 2^{d/2} h_{\ell}   + |\hat{f}_{{\ell}}^{\square}(x)|
            \\&\leq
                 \hat{\varepsilon}_{{\ell}}^{\square}(y)
                + \hat{\varepsilon}_{{\ell}}^{\square}(x)
                + K 2^{d/2} h_{\ell}   + a
        \end{align*}
        where we write $\hat{\varepsilon}_{{\ell}}^{\square}(x) = \left| \hat{f}_{{\ell}}^{\square}(x) - f(x) \right|$ for brevity.
        Finally, we utilise the triangle inequality to note that
        $
                \left| f(y) \right|
            \leq
                |\hat{f}_{{\ell}}^{\square}(y)| + \hat{\varepsilon}_{{\ell}}^{\square}(y)
        $ and see that we must have that
        \begin{align*}
                 \left| f(y) \right|
            &\leq
                 2 \hat{\varepsilon}_{{\ell}}^{\square}(y)
                + \hat{\varepsilon}_{{\ell}}^{\square}(x)
                + K 2^{d/2} h_{\ell}   + a
        \end{align*}
        for all $y \myin \square$.
        Thus for some arbitrary value $\psi > 0$, we have that any $y \in \square$ is such that
        \begin{align*}
                \mathbb{E} \!\left[
                    \sup_{x \in \square}
                        \ind{|\hat{f}_{{\ell}}^{\square}(x)|\leq a}
                    \right]
            &\leq
                \mathbb{E} \!\Big[
                        \sup_{x \in \square}
                            \ind{|\hat{f}_{{\ell}}^{\square}
                            (x)|\leq a}
                            \ind{|f(y)|\leq 2\hat{\varepsilon}_{{\ell}}^{\square}(y)
                                + \hat{\varepsilon}_{{\ell}}^{\square}(x)
                                + K 2^{d/2} h_{\ell}   + a}
                    \Big]
            \\&\leq
                \mathbb{E} \!\left[
                        \sup_{x \in \square}
                        \ind{|f(y)|\leq
                             2\hat{\varepsilon}_{{\ell}}^{\square}(y)
                                + \hat{\varepsilon}_{{\ell}}^{\square}(x)
                                + K 2^{d/2} h_{\ell}   + a}
                    \right]
            \begin{details}
             \\&=
                \mathbb{E} \!\left[
                        \ind{|f(y)|\leq
                             2\hat{\varepsilon}_{{\ell}}^{\square}(y)
                                + \inf_{x \in \square} \hat{\varepsilon}_{{\ell}}^{\square}(x)
                                + K 2^{d/2} h_{\ell}   + a}
                    \right]
            \\&\leq
                \mathbb{E} \!\left[
                        \ind{2\hat{\varepsilon}_{{\ell}}^{\square}(y)\geq
                            \psi}+
                        \ind{\inf_{x \in \square} \hat{\varepsilon}_{{\ell}}^{\square}(x)\geq
                                |f(y)|
                                - \psi
                                - K 2^{d/2} h_{\ell} - a}
                    \right]
            \end{details}
            \\&\leq
                    \mathbb{E} \!\Big[
                        \ind{2\hat{\varepsilon}_{{\ell}}^{\square}(y)\geq
                            \psi}
                +
                        \ind{\inf_{x \in \square} \hat{\varepsilon}_{{\ell}}^{\square}(x)\geq
                                \psi}
                +
                        \ind{|f(y)|\leq 2\psi + K 2^{d/2} h_{\ell}   + a}
                    \Big]
            \begin{details}
            \\&\leq
                    \mathbb{P} \!\left[
                        2\hat{\varepsilon}_{{\ell}}^{\square}(y)\geq
                                \psi
                    \right]
                +
                    \mathbb{P} \!\left[
                        \inf_{x \in \square}
                        \hat{\varepsilon}_{{\ell}}^{\square}(x)\geq
                                \psi
                    \right]
                +
                    \ind{|f(y)|\leq 2\psi  + K 2^{d/2} h_{\ell}   + a}
            \end{details}
            \\&\leq
                    \psi^{-p}
                    \mathbb{E} \!\left[
                        \left| 2\hat{\varepsilon}_{{\ell}}^{\square}(y) \right|^p
                    \right]
                +
                    \psi^{-p}
                    \mathbb{E} \!\left[
                        \left| \inf_{x \in \square} \hat{\varepsilon}_{{\ell}}^{\square}(x) \right|^p
                    \right]
                +
                    \ind{|f(y)|\leq 2\psi + K 2^{d/2} h_{\ell}   + a}
            \\&\leq
                    (2^p +1)
                    \psi^{-p}
                    \mathbb{E} \!\left[
                        \left| \hat{\varepsilon}_{{\ell}}^{\square}(y) \right|^p
                    \right]
                    +
                    \ind{|f(y)|\leq 2\psi + K 2^{d/2} h_{\ell}   + a}
        \end{align*}
        where we used Markov's inequality in the penultimate step.
        We find that
        \begin{align*}
                \sum_{\square \in U_\ell}
                    \mathbb{E} \!\left[
                    \sup_{x \in \square}
                        \ind{|\hat{f}_{{\ell}}^{\square}(x)|\leq a}
                    \right]
            &=
                h_{\ell}^{-d} \sum_{\square \in U_\ell}
                    \int_{\square}
                    \mathbb{E} \!\left[
                    \sup_{x \in \square}
                        \ind{|\hat{f}_{{\ell}}^{\square}(x)|\leq a}
                    \right]
                    \,d\mu(y)
            \\&\leq
                h_{\ell}^{-d} \sum_{\square \in U_\ell}
                    \int_{\square}
                        (2^p +1)
                        \psi^{-p}
                        \mathbb{E} \!\left[
                            \left| \hat{\varepsilon}_{{\ell}}^{\square}(y) \right|^p
                        \right]
                    +
                        \ind{|f(x)|\leq 2\psi + K 2^{d/2} h_{\ell}   + a}
                    \,d\mu(y)
            \\&\leq
                h_{\ell}^{-d}
                \left( 2^p + 1 \right) \psi^{-p} h^{\alpha p}_{\ell}
                +
                h_{\ell}^{-d}
                    \int_{\bar{D}}
                        \ind{|f(y)|\leq 2\psi + K 2^{d/2} h_{\ell}   + a}
                    \,d\mu(y)
            \\&\leq
                h_{\ell}^{-d}
                \left( 2^p + 1 \right) \psi^{-p} h^{\alpha p}_{\ell}
                +
                h_{\ell}^{-d} \rho_0 ( 2\psi + K 2^{d/2} h_{\ell} + a )
        \end{align*}
        whenever $2\psi + K 2^{d/2} h_{\ell} + a < \delta$, where $\delta$ is as in \cref{assumption:f_int_bound} and we used \cref{proposition:epsilon_ell-alpha_rate_sufficent_condition}. Now, to get all the terms to be of the same rate we can choose $\psi = c 2^{-\ell \left( \frac{d + \alpha p}{p + 1} \right)}$, where $c>0$ is chosen so that $2\psi + K 2^{d/2} h_{\ell} + a < \delta$ for any $\ell$. With this choice we obtain the claimed bound.
    \end{proof}
\end{corollary}

    One can see that by instead using cell-wise approximations we obtain, in expectation, $\mathcal{O}( h_{\ell}^{d-\alpha_p })$ extra cells due to approximation of $f$. This directly corresponds to an increase in work required to estimate the target level-set when compared to just utilising the true values of the indicator. Note that when $\alpha_p \geq 1$ this term is bounded by $\mathcal{O}(h_{\ell}^{-(d-1)})$ so that this term should not affect the overall work rate.

Making use of the above bounds on the number of cells for a given level of uniform refinement, the following lemma shows that we can obtain optimal works rates (in the sense of \cref{remark:work_optimal}), as long as the sequence of refinement criteria $\{a_{\ell}\}_{\ell=\ell_0}^{{L} - 1}$ is chosen correctly.

\begin{lemma} \label{lemma:work_bound_full}
    Assume that $f$ is Lipschitz on $\bar{D}$ with Lipschitz constant $K$, the assumptions of \cref{proposition:epsilon_ell-alpha_rate_sufficent_condition} hold, and \cref{assumption:f_int_bound} holds for fixed $\delta, \rho_0 > 0$.
    Further, assume that the sequence of non-negative refinement criteria $\{a_{\ell}\}_{\ell=\ell_0}^{{L} -1}$ is chosen such that for each $\ell \myin \{\ell_{0}, \ldots, {L} -1\}$
    \begin{align} \label{eq:a_work_upper_bound}
            a_{\ell}
        \leq
                c_2' h_{\ell}^{1 - \alpha}
    \end{align}
    for some constant $ c_2' > 0$, and a given $\ell_{0} \in \mathbb{N}$.
    Then the expected work of using \cref{algo:main_algo} to refine a uniform refinement $U_0$ at most ${L}$ additional times can be bounded as follows
    \begin{align}
            \sum_{\square \in U_0}
            \mathbb{E}[W_{0}^{\square}]
        \leq
            b h_{\ell_{0}}^{-(\frac{\alpha}{\beta} + d)} + c h_{{L}}^{- (\frac{\alpha}{\beta} + d - 1)  }
    \end{align}
    for some constants $b,c>0$ independent of ${L}$.
    \begin{proof}
        We note that the algorithm refines uniformly until we reach level $\ell = \ell_{0}$. Thus, taking $M_\ell = M_0 h_{\ell}^{-\frac{\alpha}{\beta}}$, under our definition of the work we have that
        \begin{align*}
                \sum_{\square \in U_0}
                    \mathbb{E}[W_{0}^{\square}]
            \,=\,
                N M_0 \left( \sum_{k = 0}^{\ell_0 - 1} h_{k}^{-(\frac{\alpha}{\beta} + d)} \right)
                +
                \sum_{\square \in U_{\ell_{0}} }
                \mathbb{E}[W_{\ell_{0}}^{\square}]
            \,\leq\,
                b h_{\ell_{0}}^{-(\frac{\alpha}{\beta} + d)}
                +
                \sum_{\square \in U_{\ell_{0}} }
                \mathbb{E}[W_{\ell_{0}}^{\square}]
        \end{align*}
        for some constant $b>0$ independent of $L$,
        where we bounded the first sum using the fact that $h_{\ell}$ is geometrically decreasing in $\ell$.
        For brevity, let $\gamma_\ell =  h_{\ell+1}^{- d} h_{\ell}^{d}$, which is the number of cells obtained from refining a cell at level $\ell$. Further, we note that the number of cells obtained from a refinement is bounded so that $\exists \gamma \in \mathbb{N}$ such that $\gamma_\ell \leq \gamma$ for each $\ell$.
        By expanding the definition of the work on a cell $W_\ell^{\square}$ from \cref{equation:def_work_recursive}, we find that
        \begin{align*}
                \sum_{\square \in U_{\ell_{0}} }
                \mathbb{E}[W_{\ell_{0}}^{\square}]
            &\leq
                b h_{\ell_{0}}^{-(\frac{\alpha}{\beta} + d)} +
                    N M_0
                    \sum^{{L} -1}_{\ell=\ell_{0}}
                        \gamma_{\ell} h_{\ell}^{-\frac{\alpha}{\beta} }
                        \sum_{\square_\ell \in U_{\ell}}
                            \mathbb{E}\Bigg[ \sup_{x\in\square_\ell} \ind{|\hat{f}^{\square_\ell}_{\ell}(x)|<a_{\ell} h^{\alpha}_{\ell}}\Bigg]
            \begin{details}
            \\&\leq
                b h_{\ell_{0}}^{-(\frac{\alpha}{\beta} + d)} +
                    N  c_1
                    \sum^{{L} -1}_{\ell=\ell_{0}}
                         \gamma_{\ell}
                         h_{\ell}^{-(\frac{\alpha}{\beta} + d-1)}
                +
                     N c_2
                    \sum^{{L} -1}_{\ell=\ell_{0}}
                         \gamma_{\ell}
                         h_{\ell}^{-(\frac{\alpha}{\beta} +d)}
                             (a_{\ell} h^{\alpha}_{\ell})
                \\&\hspace{2.5cm}+
                    N c_3
                    \sum^{{L} -1}_{\ell=\ell_{0}}
                         \gamma_{\ell}
                         h_{\ell}^{\alpha  \left( \frac{p}{p + 1} \right) -(\frac{\alpha}{\beta} + d)}
            \end{details}
            \\&\leq
                b  h_{\ell_{0}}^{-(\frac{\alpha}{\beta} + d)} +
                     c_1 h_{{L}}^{-(\frac{\alpha}{\beta} + d-1)}
                +
                     c_2
                    \sum^{{L} -1}_{\ell=\ell_{0}}
                         h_{\ell}^{\alpha  -(\frac{\alpha}{\beta} +d)}
                             a_{\ell}
                +
                    c_3
                    \sum^{{L} -1}_{\ell=\ell_{0}}
                         h_{\ell}^{\alpha  \left( \frac{p}{p + 1} \right) -(\frac{\alpha}{\beta} + d)}
        \end{align*}
        for some constants $c_1, c_2, c_3 > 0$ independent of $L$,
        where we used \cref{corollary:num_cells_fhat_ell} (assuming $h^{\alpha}_{\ell} a_{\ell} < \delta - K 2^{d/2} h_{\ell }$), and we bounded the second term using the fact that $h_{\ell}$ is geometrically decreasing in $\ell$. 
        In the case when $h^{\alpha}_{\ell} a_{\ell} < \delta - K 2^{d/2} h_{\ell }$ is not satisfied at $\ell = \ell_0,\ldots \ell_0'$, and since 
        $h^{\alpha}_{\ell} a_{\ell} \leq c_2' h_\ell$ and $h_\ell$ is geometrically decreasing, the condition will hold for large enough $\ell_0'$ independently of $\ell_0$ and $L$, hence the partial sum up to this level will be bounded by the other terms which increase with $L$.
        In order to bound the term involving $c_3$, we note that
        \begin{align*}
            \sum^{{L} -1}_{\ell=\ell_{0}}
                         h_{\ell}^{\alpha \left( \frac{p}{p + 1} \right) -(\frac{\alpha}{\beta} + d)}
            \leq
                \begin{cases}
                        c_3^{+} h_{{L}}^{\alpha  \left( \frac{p}{p + 1} \right) -(\frac{\alpha}{\beta} + d)}
                        &
                        \text{if } (\frac{\alpha}{\beta} + d) >  \alpha \left( \frac{p}{p + 1} \right)
                    \\
                        {L} - \ell_{0}
                        &
                        \text{if } (\frac{\alpha}{\beta} + d) \leq  \alpha  \left( \frac{p}{p + 1} \right)
                \end{cases}
        \end{align*}
        for some constant $c_3^{+} > 0$ independent of ${L}$. Since $\alpha_p > 1$ this term is thus bounded above by at most $c_{3}^{*} h_{{L}}^{- (\frac{\alpha}{\beta} + d - 1)}$ for some constant $c_3^{*} > 0$ independent of ${L}$, and thus matches the required rate.

        We are able to control the term involving $c_2$ via our choice of refinement criteria. In particular, for each $\ell \myin \{\ell_{0}, \ldots, {L} -1\}$, the criteria $a_{\ell}$ satisfying
        \begin{align*}
                \sum^{{L} -1}_{\ell=\ell_{0}}
                    h_{\ell}^{\alpha  -(\frac{\alpha}{\beta} +d)} a_{\ell}
            &\leq
                c_2^{*} h_{{L}}^{ \min\left\{ 1, \alpha\left(\frac{p}{p+1}\right)\right\} - (\frac{\alpha}{\beta} + d)  }
        \end{align*}
        for some constant $c_2^{*} > 0$ will ensure the optimal work rate is attained.
        %%%
        Letting $\bar{\alpha}_p = \min\left\{ 1, \alpha\left(\frac{p}{p+1}\right) \right\}$ for brevity, we can rewrite the sum on the LHS of the above as
        \begin{align*}
                \sum^{{L} -1}_{\ell=\ell_{0}}
                    h_{\ell}^{\bar{\alpha}_p - (\frac{\alpha}{\beta} + d)  }
                    \left(a_{\ell} h_{\ell}^{\alpha - \bar{\alpha}_p } \right)
        \end{align*}
        to see that \cref{eq:a_work_upper_bound} ensures that he optimal work rate is attained, since $\alpha_p > 1$, concluding the proof.
    \end{proof}
\end{lemma}

\cref{lemma:work_bound_full} ensures that our adaptive refinement scheme for level-set approximation is work-optimal, in the sense that the work rate is the same as the rate obtained by setting $a=0$ in \cref{lemma:num_cells_f<a}, which assumes exact knowledge of $f$.

%% file: analysis_error.tex
\subsection{Error Analysis} \label{section:error}

In this section, we focus on bounding the error of our adaptive sampling scheme for level-set estimation as given in \cref{algo:main_algo}. Before we are able to perform the error analysis, we require some notion of distance between the target level set $\mathcal{L}_0$ and the approximated level-sets generated by our adaptive sampling scheme. To that end, consider a (not necessarily uniform) collection of cells $\mathcal{R}_\ell$ for which the smallest cell-size is $h_\ell$. For such a collection of cells with cell-wise functional approximations $\hat{f}^{\square}_{\ell}(x)$ for each cell $\square \in \mathcal{R}_\ell$, we define the true and approximated sets
\begin{align*}
    \mathcal{L}_{\mathrm{in}} := 
        \left\{ 
            x \in \bar{D} 
        \,\, \Big\vert \,\,
            f(x) \leq 0
        \right\}
    ; \,\,\,\,\,\,\,\,
        \hat{\mathcal{L}}_{\ell,\mathrm{in}}^{\square} := 
            \left\{ 
                x \in \square
            \,\, \Big\vert \,\,
                \hat{f}^{\square}_{\ell}(x) \leq 0
            \right\}
    ; \,\,\,\,\,\,\,\,
        \hat{\mathcal{L}}_{\ell,\mathrm{in}} := \bigcup_{\square \in \mathcal{R}_\ell} \hat{\mathcal{L}}_{\ell,\mathrm{in}}^{\square}
\end{align*}
where each set $\hat{\mathcal{L}}_{\ell,\mathrm{in}}^{\square}$ is disjoint so that, for any $x \myin \hat{D}$,
\begin{align*}
    \mathbb{P}[ x \in \hat{\mathcal{L}}_{\ell,\mathrm{in}}] 
        = 
            \sum_{\square \in \mathcal{R}_\ell} \ind{x \in \square} \cdot \mathbb{P}[ x \in \hat{\mathcal{L}}_{\ell,\mathrm{in}}]
        =
            \sum_{\square \in \mathcal{R}_\ell} \mathbb{P}[ x \in \hat{\mathcal{L}}_{\ell,\mathrm{in}}^{\square}].
\end{align*}

With these sets defined, we take as an error metric the following random function
\begin{equation} \label{def:error_DELTAx}
    \Delta_\ell(x) := 
        \left\vert \ind{x \in\mathcal{L}_{\mathrm{in}}} - \ind{x \in  \hat{\mathcal{L}}_{\ell,\mathrm{in}}} \right\vert
\end{equation}
which identifies any point $x \myin \hat{D}$ where there is a difference in sign (taking 0 as negative) of the target function $f$ and the estimate $\hat{f}_{\ell}^{\square}$ corresponding to a cell with $\square \in \mathcal{R}_\ell$.
From this, take the following measure of the accuracy of \cref{algo:main_algo}
\begin{equation}
    \label{eq:error_metric}
    \int_{\bar{D}}
                \Delta_\ell(x)
            d\mu(x)
\end{equation}
which measures the volume of the (random) region of the domain $\bar{D}$ where $f$ and $\hat{f}_{\ell}^{\square}$ have different sign.
This error measure can be localised. Indeed, it holds that
\begin{align*}
        \int_{\bar{D}}
            \mathbb{E}[\Delta_\ell(x)]
        d\mu(x)
    &=
        \sum_{\square \in \mathcal{R}_\ell}
            \int_{\square}
                 \mathbb{E} \left[ \Delta_\ell^{\square}(x) \right]
            d\mu(x)
\end{align*}
where we assume $\mathcal{R}_\ell$ contains the whole domain, and
\(
    \Delta_{\ell}^{\square}(x) := 
        \left\vert \ind{x \in\mathcal{L}_{\mathrm{in}}} - \ind{x \in  \hat{\mathcal{L}}_{\ell,\mathrm{in}}^{\square}} \right\vert
\)
is the cell-wise version of \cref{def:error_DELTAx}.
Hence, it is enough to compute the error cell-wise using the metric
\begin{equation*}
    \int_{\square}
            \mathbb{E}\left[
                \Delta_\ell^{\square}(x)
            \right]
            d\mu(x)
\end{equation*}
over each cell $\square \in \mathcal{R}_\ell$. 
\begin{remark} \label{remark:error_d-1}
    Since we are actually interested in the points in $\square$ for which the zero level-set of the target and estimated functions differ, one might consider instead using 
    $\Delta_\ell(x) = \left\vert \ind{f(x)=0} - \ind{\hat{f}_{\ell}^{\square}(x)=0} \right\vert$. 
    The issue with this choice is that since $\mathcal{L}_0$ is of Hausdorff dimension $d - 1$, picking any function such that $\hat{f}_{\ell}^{\square}(x) \neq 0$ $\forall x \in \square$ means that $\int_{\square} \Delta_\ell(x) d\mu(x) = \int_{\square} \ind{f(x)=0} d\mu(x) = 0$, rendering the error metric uninformative. Our choice of $\Delta_\ell(x)$ as in \cref{def:error_DELTAx} does not suffer from this limitation provided that $f$ changes sign across the level-set.
\end{remark}

With this setup, we define the (expected) error of the method on a cell $\square \in \mathcal{R}_\ell$ in a similar way as for the work: by the recursive formula
\begin{align} \label{eq:error_metric_recursive}
        \mathbb{E}[
            E_{\ell}^{\square}
        ]
    &= 
            \int_{\square}
            \mathbb{E}\left[
                \ind{\hat{\delta}_{\ell}^{\square} \geq a_\ell}
            \Delta_\ell^{\square}(x)
            \right]
            d\mu(x)
        +
            \sum_{\square' \in \mathcal{R}(\square)}
            \mathbb{E}\left[
                \ind{\hat{\delta}_{\ell}^{\square} < a_\ell}
                E_{\ell+1}^{\square'}
            \right]
\end{align}
where we swapped the order of integration by Fubini's theorem. Note that in the second term the indicator is still over the cell $\square \in \mathcal{R}_\ell$ but we consider the error of the refined cells $\square' \in \mathcal{R}(\square)$ instead.

The following proposition establishes that a level-$\ell$ global uniform refinement has an integrated expected error of $\mathcal{O}\left(h_\ell^{\alpha_p}\right)$ in our chosen error metric \cref{eq:error_metric_recursive}.

\begin{proposition} \label{proposition:error_bound_non_adaptive}
    Let the assumptions of \cref{proposition:epsilon_ell-alpha_rate_sufficent_condition} hold. Then, a level-$\ell$ uniform refinement $U_\ell$ satisfies
    \begin{align*}
    \sum_{\square \in U_\ell}
        \int_{\square}
            \mathbb{E}[\Delta^{\square}_\ell(x)] 
            d\mu(x)
        \leq 
            c_{\Delta} h_\ell^{\alpha_{p}}
    \end{align*}
    for some constant $c_{\Delta}>0$ independent of the level $\ell$.
    \begin{proof}
        Using the definition of our error, we find that 
        \begin{align*}
            \mathbb{E}[\Delta_{\ell}^{\square}(x)] 
        &\leq 
            \mathbb{E}\left[\ind{|f(x) - \hat{f}_\ell^{\square}(x)| \geq |{f}(x)|}\right]
        =
            \mathbb{E}\left[\ind{\hat{\varepsilon}_\ell^{\square}(x) \geq |{f}(x)|}\right]
        \end{align*}
        where we write $\hat{\varepsilon}_\ell^{\square}(x) = |f(x) - \hat{f}_\ell^{\square}(x)|$ for brevity. Now, for any choice of $\psi$ we have that
        \begin{align*}
            \mathbb{E}[\Delta_{\ell}^{\square}(x)] 
        &\leq 
            \mathbb{E}\left[
                \ind{|\hat{\varepsilon}_\ell^{\square}(x)| \geq |f(x)|}
                \ind{|f(x)| < \psi}
            \right]
            +
            \mathbb{E}\left[
                \ind{|\hat{\varepsilon}_\ell^{\square}(x)| \geq |f(x)|}
                \ind{|f(x)| \geq \psi}
            \right]
        \\&\leq 
            \ind{|f(x)| < \psi}
            +
            \mathbb{P}\left[
                |\hat{\varepsilon}_\ell^{\square}(x)| \geq \psi
            \right]
        \begin{details}
        \\&\leq 
            \ind{|{f}(x)| < \psi}
            + 
            \mathbb{P}\left[
                |f(x) - \hat{f}_\ell^{\square}(x)| \geq \psi
            \right]
        \end{details}
        \\&\leq 
            \ind{|{f}(x)| < \psi}
            + 
            \psi^{-p}
                \mathbb{E}\left[\left|f(x) - \hat{f}_\ell^{\square}(x)\right|^{p}\right]
        \end{align*}
        where we used the Markov inequality in the final step.
        Taking the integral over the given cell and summing over all cells of that size, we obtain
        \begin{align*}
            \sum_{\square \in U_\ell}
                    \int_{\square}
                        \mathbb{E}[\Delta^{\square}_\ell(x)] 
                    d\mu(x)
            &\leq
                \rho_0  \psi^{} 
                + \psi^{-p} h_{\ell}^{\alpha p}
        \end{align*}
        where we used \cref{assumption:f_int_bound,proposition:epsilon_ell-alpha_rate_sufficent_condition}.
        Since the above holds for any choice of $\psi$, we can choose $\psi = \min\left\{ 1, \delta \right\} h_\ell^{\alpha_p}$ to get the final two terms to be of the same order, completing the proof.
    \end{proof}
\end{proposition}

Building on the previous result, the following lemma tells us that, for certain choices of refinement criteria sequences $\{a_{\ell}\}_{\ell=\ell_0}^{{L} - 1}$, the result of our adaptive algorithm has an integrated expected error of $\mathcal{O}\left(h_{{L}}^{\alpha_p}\right)$ in our chosen error metric.

\begin{lemma} \label{lemma:full_error_bound}
    Let the assumptions of \cref{proposition:epsilon_ell-alpha_rate_sufficent_condition} hold.
    Then for any ${L}\myin \mathbb{N}$ with $\ell_{0} < {L}$, choosing the sequence of a refinement criteria $\{a_{\ell}\}_{\ell=\ell_0}^{{L} - 1}$ such that
    \begin{equation} \label{eq:a_error_lower_bound}
                \sum_{\ell = \ell_{0}}^{{L} - 1}
                    a_{\ell}^{-p}
            <
                c'
                h_{{L}}^{\alpha_p}
    \end{equation}
    for some constant $c'>0$ independent of $L$, ensures that the expected error obtained from using \cref{algo:main_algo} on a uniform refinement $U_{\ell_{0}}$ satisfies
    \begin{align*}
            \sum_{\square \in U_0}
            \mathbb{E}[
                E_{0}^{\square}
            ]
        \leq
            c h_{{L}}^{\alpha_p}
    \end{align*}
    for some constant $c>0$.
    \begin{proof}
        For a given level $\ell$ write
        $\varepsilon_{\ell}^{\square} = 
            \left(
            h_{\ell}^{-d}
            \int_{\square} 
                \mathbb{E}\left[\left| f(x) - \hat{f}^{\square}_{\ell} \right|^p\right] 
            d\mu(x)
            \right)^{1/p}
        $
        for each $\square \in U_\ell$ where $U_\ell$ is a level-$\ell$ uniform refinement. We note that
        \begin{align*}
                \mathbb{E}\left[
                    \ind{\hat{\delta}_{\ell}^{\square} \geq a_\ell}
                        \Delta_\ell^{\square}(x) 
                    \right]
            &=
                \mathbb{E}\left[
                    \ind{\hat{\delta}_{\ell}^{\square} \geq a_\ell}
                        \left\vert \ind{f(x) \leq 0} - \ind{\hat{f}_{\ell}^{\square}(x) \leq 0} \right\vert
                    \right]
            \\&\leq
                    \mathbb{E}\left[
                        \ind{\inf_{x \in \square}|\hat{f}_{\ell}^{\square}(x)| \geq a_\ell \varepsilon^{\square}_{\ell}}
                        \ind{|f(x) - \hat{f}_{\ell}^{\square}(x)| > |\hat{f}_{\ell}^{\square}(x)|}
                    \right]
            \\&=
                    \mathbb{P}\left[
                        |f(x) - \hat{f}_{\ell}^{\square}(x)| > a_\ell \varepsilon_{\ell}^{\square}
                    \right]
            \\&\leq
                    \min \left\{ 1, \,
                        \left(a_\ell \varepsilon_{\ell}^{\square} \right)^{-p}
                        \mathbb{E}\left[\left|
                        f(x) - \hat{f}_{\ell}^{\square}(x)
                        \right|^{p}\right]
                    \right\}
        \end{align*}
        where we used the Markov inequality in the final step. Thus we obtain that
        \begin{align} \label{eq:error_bound_not_refined}
            \sum_{\square \in U_\ell}
            \int_{\square}
                \mathbb{E}\left[
                    \ind{\hat{\delta}_{\ell}^{\square} \geq a_\ell}
                \Delta_\ell^{\square}(x)
                \right]
                d\mu(x)
            &\leq 
                a_\ell^{-p}
        \end{align}
        by the definition of $\varepsilon_{\ell}^{\square}$.
        Thus, the integrated expected error of a cell which our method does not refine at a level $\ell$ can be bounded using the refinement criteria $a_\ell$.
        With this in hand, we expand the error to find
        \begin{align*}
                \sum_{\square \in U_{\ell_{0}} }
                \mathbb{E}[
                    E_{\ell_{0}}^{\square}
                ]
            &\leq
                \sum^{{L} -1}_{\ell=\ell_{0}}
                \left(
                \sum_{\square_{\ell} \in U_{\ell}}
                    \int_{\square_{\ell}}
                        \mathbb{E}\left[
                        \Delta_{\ell}^{\square_{\ell}}(x)
                        \right]
                        d\mu(x)
                \right)
                +
                \left(
                \sum_{\square_{{L}} \in U_{{L}}}
                    \int_{\square_{{L}}}
                        \mathbb{E}\left[
                        \Delta_{{L}}^{\square_{{L}}}(x)
                        \right]
                        d\mu(x)
                \right)
            \\&\leq
                    \left(
                    \sum_{\ell = \ell_{0}}^{{L} - 1}
                        a_{\ell}^{-p}
                    \right)
                    + c_{\Delta} h_{{L}}^{\alpha_{p}}
        \end{align*}
        where we used \cref{eq:error_bound_not_refined,proposition:error_bound_non_adaptive}.
        Now, under our assumptions, we have that \cref{eq:a_error_lower_bound} holds and thus obtain the requested error rate with $c = \max\{c', c_{\Delta}\}$.
    \end{proof}
\end{lemma}

%% file: criteria.tex
\subsection{Selection of the Refinement Criteria} \label{subsection:refinement_criteria}

\cref{algo:main_algo} depends on a few user-specified hyperparameters. Whilst some of these are typically known a priori or can be easily estimated (e.g. $\alpha, \beta$) others such as the sequence of refinement criteria $\{a_{\ell}\}_{\ell = \ell_0}^{{L}}$, the maximum number of times to refine ${L}$ a cell, and the base uniform refinement $U_0$ to employ the algorithm on are less trivial to select.
In this subsection we suggest some simple choices for these parameters.

For a choice of optimal refinement criteria, we adapt the criteria utilised in \cite{haji2022adaptive} to obtain a choice of the form
\begin{equation} \label{eq:define_a}
        a_{\ell}
    = 
            c h_{\ell}^{\alpha_p/R} h_{{L}}^{\alpha_p (R-1)/ R} h_{\ell_{0}}^{-\alpha_p/R} h_{\ell}^{-\alpha}
\end{equation}
for some chosen constant $c>0$ and fixed $\ell_{0} \myin \mathbb{N}$,
where $\alpha_p = \alpha\left(\frac{p}{p+1}\right)$ and the parameter $R>1$ controls the strictness of refinement.

The following result ensures that for the choice of refinement criteria \eqref{eq:define_a}, optimal work and error rates are obtained whenever $1 < R < \alpha_p$.

\begin{lemma} \label{lemma:criteria_work_bound}
    For the choice of refinement criteria in \eqref{eq:define_a} with $1 < R < \alpha_p$, we have that for any ${L} \in \mathbb{N}$ choosing 
    \(
        \ell_{0} = \left\lceil {L} \left(1 - \frac{p}{R(p+1)}\right) \right\rceil
    \)
    ensures that, for any $\ell \in \{\ell_{0}, \ldots, {L} - 1\}$, \cref{eq:a_work_upper_bound,eq:a_error_lower_bound} hold for some constants.
    \begin{proof}
    We have that
        \(
                a_{\ell}
            \leq 
                c_2' h_{\ell}^{\min\left\{1, \alpha_p\right\} - \alpha} 
        \)
        for some constant $c_2'>0$, whenever $R < \min\left\{1, \alpha_p\right\}^{-1} \alpha_p = \alpha_p$ since $\alpha_p > 1$. 
        Note that \cref{eq:a_work_upper_bound} also enforces that $ a_{\ell} < (\delta - K 2^{d/2} h_{\ell}) h^{-\alpha}_{\ell}$ for some choice of the constant $c>0$ and a sufficiently refined base refinement $U_{\ell_0}$, i.e. when
        \(
                c h_{L}^{\alpha_p (R-1)/R} + h_{\ell_0} K 2^{d/2}
            <
                \delta.
        \)
        
        Since we have that $R>\left(\frac{p}{p+1}\right) = \frac{\alpha_p}{\alpha}$, $a_{\ell}^{-p}$ is geometrically decreasing in $\ell$ and we can use the bound
        \begin{align*}
                \sum_{\ell = \ell_{0}}^{{L} - 1}
                    a_{\ell}^{-p}
            &\leq
                   c^{*} a_{\ell_{0}}^{-p}
            =
                c^{*} c^{-p} h_{\ell_{0}}^{\alpha p}
                h_{{L}}^{ -\alpha_p p (R-1)/R}
        \end{align*}
        for some constant $c^{*} >0$. 
        To ensure that \cref{eq:a_error_lower_bound}
        it is enough to require that
        \begin{align*}
                c^{*} c^{-p} h_{\ell_{0}}^{\alpha p} 
            &\leq
                b
                h_{{L}}^{\alpha_p (1 + p(R-1)/R)}
            = 
                b
                h_{{L}}^{\alpha p \left(1 - \frac{p}{R(p+1)}\right)}
        \end{align*}
         for some constant $b > 0$.
        Thus we have that, for any choice of $R>\left(\frac{p}{p+1}\right)$, choosing 
        \(
                \ell_{0}
            \geq
                {L} \left(1 - \frac{p}{R(p+1)}\right)
        \)
        ensures \cref{eq:a_error_lower_bound} holds, assuming that $\{h_\ell\}_{\ell = \ell_{0}}^{{L}}$ is a geometrically decreasing sequence. This concludes the proof.
    \end{proof}

    \begin{remark} \label{remark:vartheta_general}
        The proof of \cref{lemma:criteria_work_bound} shows that in practice $\ell_{0} \in \mathbb{N}$ need only satisfy
        \(
                h_{\ell_{0}}
            \leq
                \left(\frac{c_{\Delta}}{c^{*} c^{-p}}\right)^{\frac{1}{\alpha p}}
                h_{{L}}^{\left(1 - \frac{p}{R(p+1)}\right)}
        \)
         where the constant $c_{\Delta}>0$ is as in \cref{proposition:error_bound_non_adaptive}. Thus, if $c_\Delta$ is known or can be estimated, one can directly evaluate the required value of $h_{\ell_{0}}$, i.e. the cell size to be uniformly refined to. This can be done since the other constants are a direct consequence of the refinement criteria itself ($c$ is chosen in the criteria and $c^{*}$ is typically bounded above by $4$ for valid choices of $\alpha_p$ and $R$). Note that in the case where $c_\Delta$ is unknown, the desired rates will still hold asymptotically for the choice of $\ell_0$ provided in \cref{lemma:criteria_work_bound}.
    \end{remark}
\end{lemma}

\subsection{Cost-Complexity} \label{subsection:selection_base_refinement}

From \cref{section:work,section:error}, one can calculate the cost-\\complexity of \cref{algo:main_algo} for approximating a level-set of Hausdorff dimension $d - 1$. 

\begin{theorem} \label{theorem:final_complexity}
    Let \cref{assumption:f_int_bound} hold for fixed $\delta, \rho_0 > 0$, \cref{assumption:interpolant_rate_alpha,assumption:beta,ass:continuity_of_T} hold for fixed values of $\alpha > \left(\frac{p+1}{p}\right)$, $\beta > 0$, and $C_N>0$ respectively, and assume that $f$ is Lipschitz on $\bar{D}$. Let $\varepsilon$ denote the desired error (in the error metric \cref{eq:error_metric}) of the final level-set estimation.
    Then, selecting ${L}$ such that $h_{{L}} = \mathcal{O} \left( \varepsilon^{1/\alpha_p} \right)$, and choosing the sequence of refinement criteria $\{a_{\ell}\}_{\ell=\ell_0}^{{L} - 1}$ to be as in \eqref{eq:define_a} with
    \(
        \ell_{0} = \left\lceil {L} \left(1 - \frac{p}{R(p+1)}\right) \right\rceil
    \)
    and restricting $1<R<\alpha_p = \alpha \left(\frac{p}{p+1}\right)$, ensures that  using \cref{algo:main_algo} to approximate a level-set of Hausdorff dimension $d - 1$ has a cost-complexity of
    \(
        \mathcal{O} \left( \varepsilon^{-(d-1 +{\alpha}/{\beta}) / \alpha_p} \right).
    \)
    \begin{proof}
        Follows directly from \cref{lemma:criteria_work_bound}.
    \end{proof}
\end{theorem}

By our analysis in the aforementioned sections, this rate is optimal (in the sense of \cref{remark:work_optimal}) for any adaptive scheme utilising the same grid structure and cell-refinement method. 
Similarly, one can calculate that the cost-complexity of a uniform refinement is
$
   \mathcal{O} \left( \varepsilon^{-(d + {\alpha}/{\beta} ) / \alpha_p} \right)
$
so that our adaptive refinement scheme attains a cost-complexity reduction of
$
   \mathcal{O} \left( \varepsilon^{-1/\alpha_p} \right)
$
compared to a uniform (non-adaptive) approximation.

%% file: results.tex
\section{Numerical Results} \label{section:numerics}

In this section, we present numerical results which verify that the complexity results, as well as the asymptotic work and error rates, from \cref{section:complexity_analysis} hold in practice. \Cref{subsection:results_two_functions} focuses on estimating level-sets of well-known continuous optimization functions with complex level-set structures \cite{TestFunctionIndex}, whilst \cref{subsection:numerics_failure_region_hyperelastic_beam} focuses on estimating the failure region for a hyperelastic beam. The approximation of failure probabilities often requires rare-event estimation techniques such as importance sampling. Here we do not consider a rare-event estimation problem for the sake of simplicity. Nevertheless, we remark that \cref{algo:main_algo} is perfectly compatible with rare event estimation methods.

We begin by describing the part of our setup which is common among all test
problems:

\begin{itemize}
    \item We take the refinement criterion defined in \eqref{eq:define_a1} and we employ piecewise-bilinear discontinuous Lagrange interpolants over boxed cells to approximate the target function. Consequently, we can set $\alpha=2$ for all problems. We always start with an initial uniform mesh of sufficiently small mesh size, cf. \cref{remark:vartheta_general}.
    \item When the exact solution is known (first two test problems) we compute
        the error measure \eqref{eq:error_metric} using $512$
        randomized quasi-Monte Carlo points\footnote{The scrambled Sobol' sequence implementation
        of the SciPy library,
    cf.~\url{https://docs.scipy.org/doc/scipy/reference/generated/scipy.stats.qmc.Sobol.html}.}
        per cell. To compute the expected error we run Algorithm 1 independently
        $N_{\text{MC}}$ times until the standard deviation of the estimated error is
        acceptably small.
    \item In terms of software implementation, we employ the open-source finite
    element libraries MFEM \cite{mfem} and FEniCSx \cite{baratta2023dolfinx}. We
    adopt the former for local mesh refinements and the latter for finite
    element computations and post-processing.
    %\footnote{Choice motivated by convenience.}
\end{itemize}

The open-source code for the numerical experiments is freely available on GitHub at \url{https://github.com/croci/contourest}.

\subsection{Two Functions with Complex Level-set Structures} \label{subsection:results_two_functions}

Here we consider estimating the zero level-sets of two functions with complex level-set structures. First, we consider a $2$-dimensional drop-wave function of the form
\begin{align*}
    \frac{1}{5}-\dfrac{1 + \cos(12\lVert x \rVert_2)}{\lVert x \rVert_2^2/2 + 2},\quad x \in[-5,5]^2.
\end{align*}
We also consider a (scaled and shifted) $3$-dimensional Styblinski-Tang function of the form
\begin{align*}
    \frac{1}{122}\left(\sum_{i=1}^3 x_i^4 -16x_i^2 + 5x_i\right) + 1, \quad x \in [-5,5]^3.
\end{align*}

In order to make the approximation problem stochastic, we add zero-mean independent Gaussian noise to each point evaluation of both functions. We use Gaussian noise with variances $h_{\ell}^{2}$ and $h_{\ell}^{2}/3$ for the drop-wave and Styblinski-Tang functions respectively, with corresponding grid sizes of $h_\ell=2^{-(\ell+6)}$ and $h_\ell=2^{-(\ell+2)}$. This roughly corresponds to the noise level of averaging $M_\ell=h_\ell^{\alpha/\beta}$ Monte Carlo samples of a random variable with unit and $\frac{1}{3}$ variance respectively (by the Central Limit Theorem, Monte Carlo estimates are asymptotically Gaussian). In both cases, we set $\beta=1/2$ and $p=\infty$ to represent this Monte Carlo-like setting. We compute the expected error by averaging $N_{\text{MC}}=10$ and $N_{\text{MC}} = 1024$ independent runs of the algorithm in for the 2D and 3D problem respectively.

We show the approximation and grid size resulting by setting ${L} = 6$ in \cref{algo:main_algo} in \cref{figs:dropwave_results} for the drop-wave function and \cref{figs:tang_results} for the Styblinski-Tang function. It is evident from these figures that the target level-set is qualitatively well approximated and that the resulting mesh is automatically adaptively refined only around the level set (or in areas in which the function is small). 
In \cref{fig:styblinski_total_complexity} we present more quantitative results: we plot the complexity rates of the algorithm over different values of ${L}$ for both of the considered functions. These results match the expected rates from \cref{theorem:final_complexity} and support our theoretical findings.

\begin{figure}[ht]
    \begin{subfigure}{.48\textwidth}
        \centering
        \includegraphics[height=0.75\linewidth]{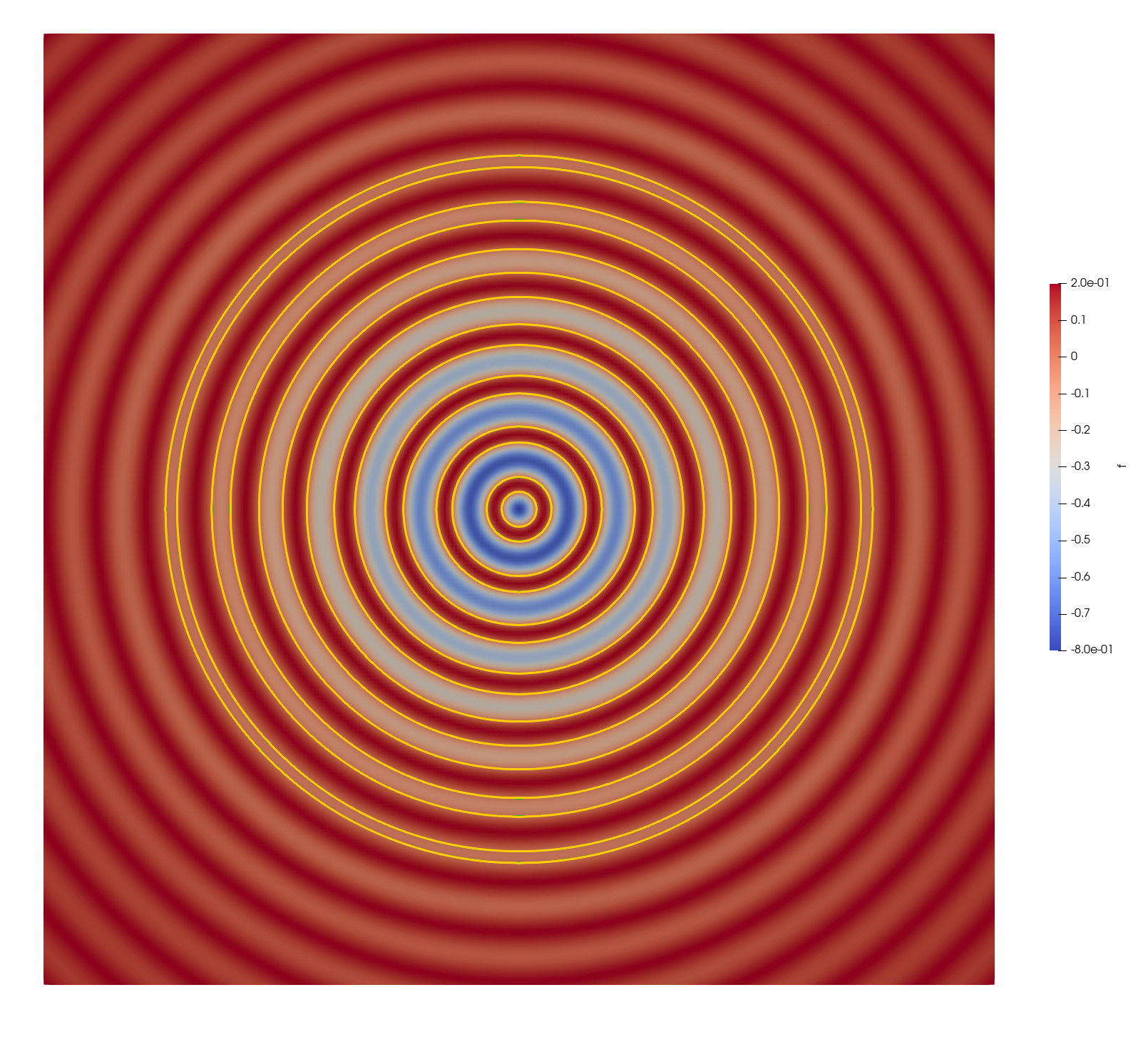}
        \caption{True (orange) and approximated (green) level-sets.}
        \label{fig:2D_level_set}
    \end{subfigure}%
    \hfill
    \begin{subfigure}{.48\textwidth}
        \centering
        \includegraphics[height=0.75\linewidth]{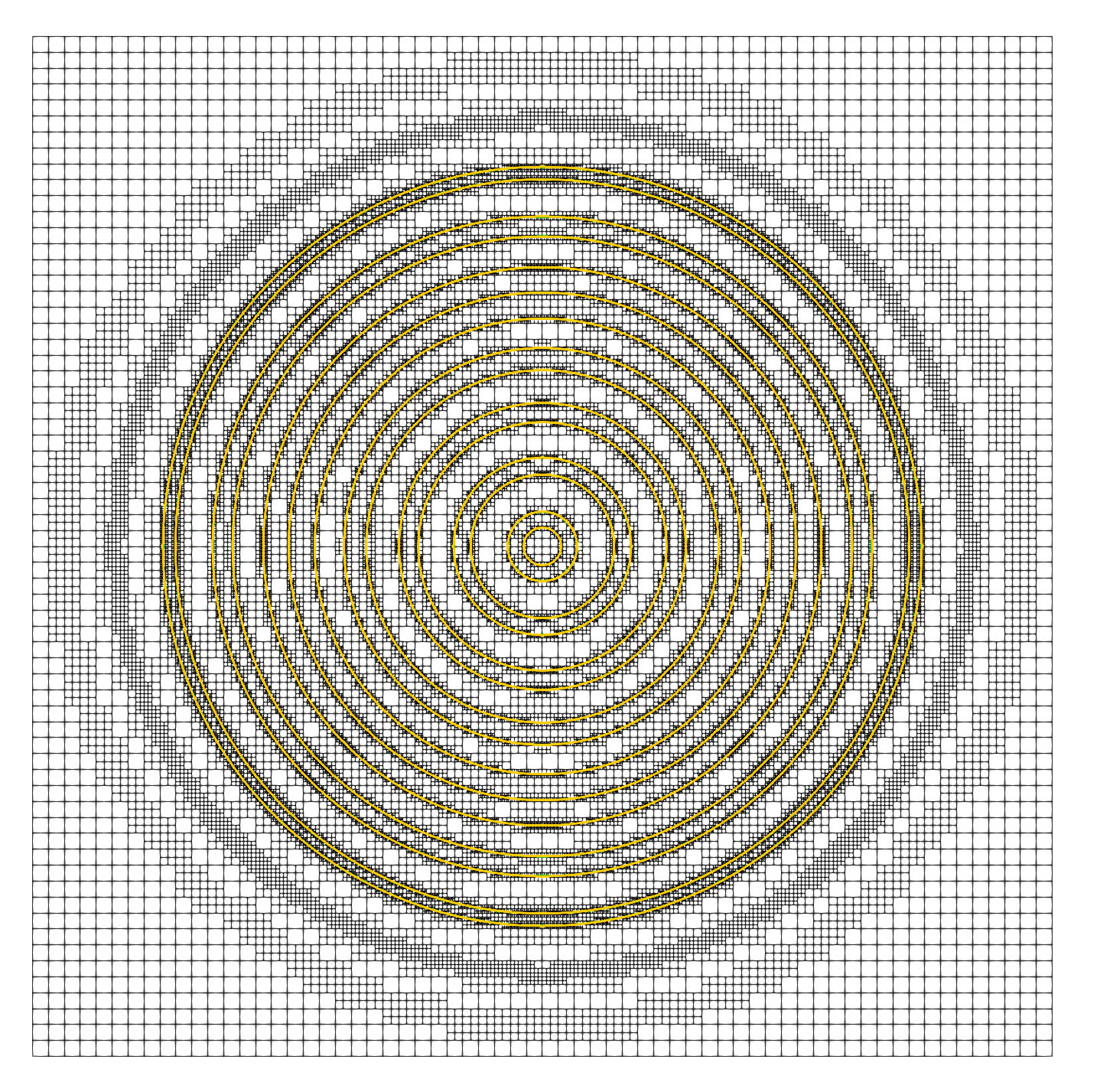}
        \caption{Final adapted grid.}
        \label{fig:2D_mesh}
    \end{subfigure}%
    \\
    \caption{Results of the adaptive algorithm for the 2-dimensional drop-wave function with ${L} = 6$. One can see heuristically that the level-set is well approximated and cells are only refined close to the target set.}
    \label{figs:dropwave_results}
\end{figure}

\pgfplotsset{
  width=0.49\textwidth,
  height=5.8cm,
  grid=major,
}

\begin{figure}[htp]
    \begin{subfigure}{.48\textwidth}
        \centering
        \includegraphics[height=0.75\linewidth]{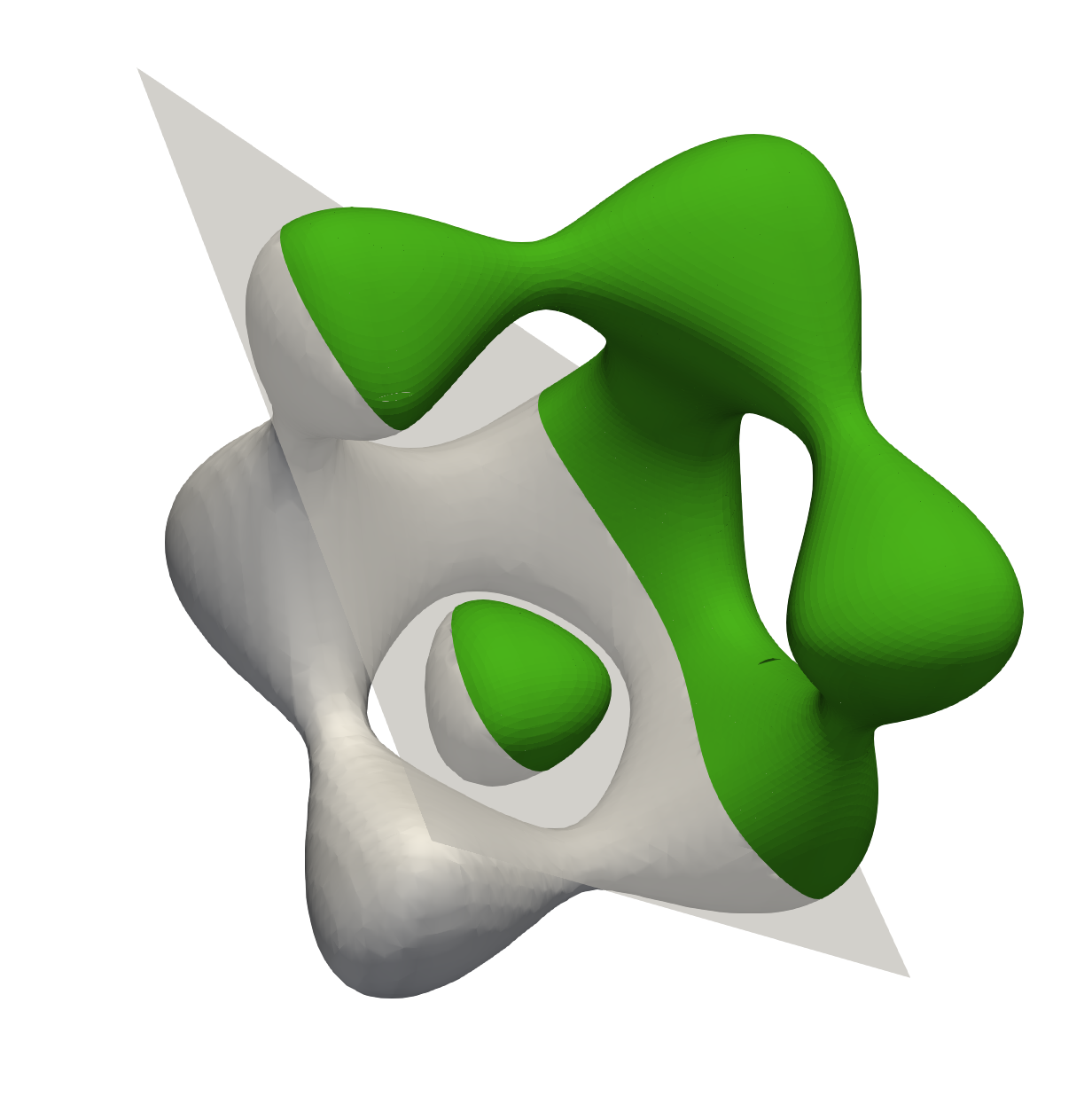}
        \caption{True (white) and approximate (green) level-sets.}
        \label{fig:3D_level_set}
    \end{subfigure}%
        \hfill
    \begin{subfigure}{.48\textwidth}
        \centering
        \includegraphics[height=0.75\linewidth]{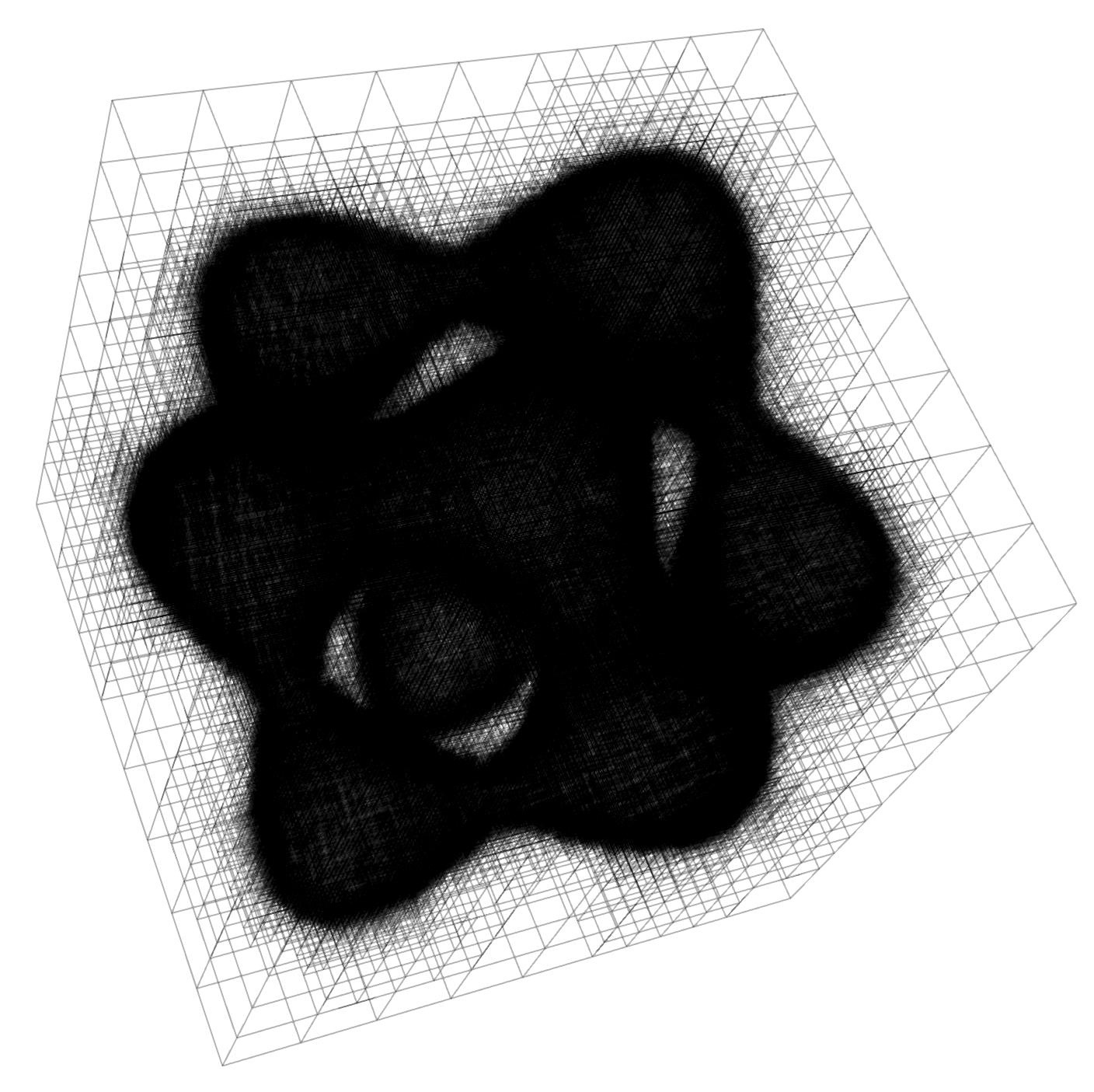}
        \caption{Final adapted grid.}
        \label{fig:3D_mesh}
    \end{subfigure}%
    \\
    \caption{Results of the adaptive algorithm for the 3-dimensional Styblinski-Tang function with ${L} = 6$. One can see heuristically that the level-set is well approximated and cells are only refined close to the target set.}
    \label{figs:tang_results}
\end{figure}

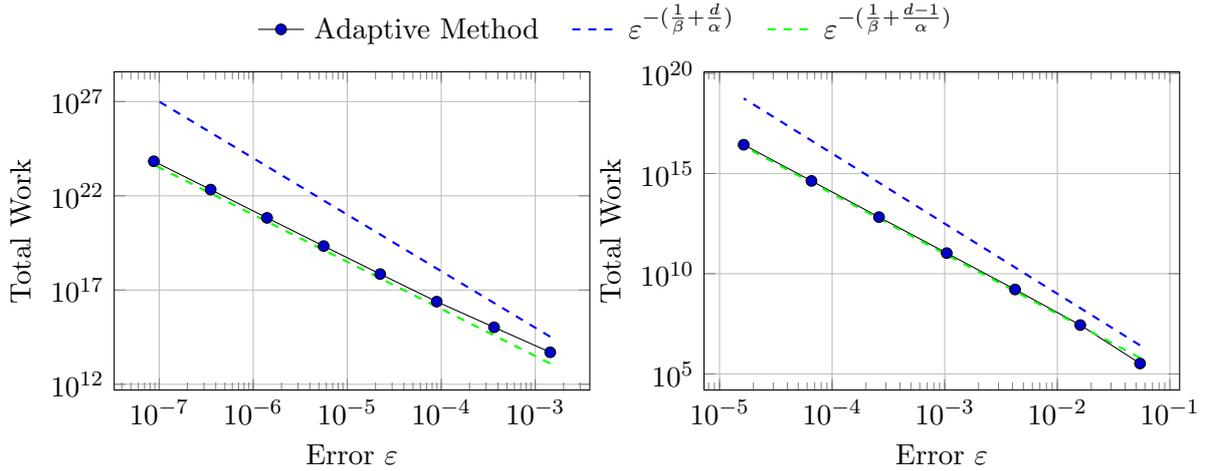
\begin{figure}[ht]
    \centering
    \begin{tabular}{ccc}
        \ref{plot:complexity_adaptive} Adaptive Method
		& \ref{plot:complexity_full_rate} \(\varepsilon^{-(\frac{1}{\beta} + \frac{d}{\alpha})}\) 
        & \ref{plot:complexity_adapt_rate} \(\varepsilon^{-(\frac{1}{\beta} + \frac{d-1}{\alpha})}\) 
	\end{tabular}
    \begin{tikzpicture}
        \begin{axis}[
            %width=0.8\textwidth,
            %height=0.45\textwidth,
            grid = major,
            xlabel={Error $\varepsilon$},
            ylabel={Total Work},
            ymode=log,
            xmode=log,
            legend style={at={(0.99,0.99)},anchor=north east,fill opacity=0.8, draw opacity=1, text opacity=1, fill=white,draw=black},
            x tick label style={/pgf/number format/fixed},
            y tick label style={/pgf/number format/sci}
        ]
        
        \addplot+[
            mark=*,
            black,
            error bars/.cd,
                y dir=both, y explicit,
        ] coordinates {
                (1.44053996e-03, 4.99336320e+13) +- (1.20976323e-06, 2.24116379e+09)
                (3.64439934e-04, 1.05470358e+15) +- (4.21522548e-07, 8.54605158e+10)
                (8.94840807e-05, 2.42346910e+16) +- (1.16264633e-07, 1.20836305e+10)
                (2.24017538e-05, 7.08246981e+17) +- (2.17362788e-08, 1.96732903e+12)
                (5.61439665e-06, 2.15952303e+19) +- (6.86972821e-09, 8.63075724e+13)
                (1.40444899e-06, 6.74971144e+20) +- (2.33060905e-09, 4.20572652e+14)
                (3.52824281e-07, 2.13254165e+22) +- (4.23010184e-10, 1.64345823e+15)
                (8.77202183e-08, 6.78732568e+23) +- (1.37845088e-10, 5.12596299e+17)
        };
            %\addlegendentry{Adaptive Method}
            \label{plot:complexity_adaptive}

            \addplot [thick, dashed, draw=blue, domain=1.44053996e-03:8.77202183e-08] {10^(6) * x^(-(2+1))};
            %\addlegendentry{\(\varepsilon^{-(\frac{1}{\beta} + \frac{d}{\alpha})}\)}
            \label{plot:complexity_full_rate}

            \addplot [thick, dashed, draw=green, domain=1.44053996e-03:8.77202183e-08] {10^(6) * x^(-(2+ (2-1)/2 ))};
            %\addlegendentry{\(\varepsilon^{-(\frac{1}{\beta} + \frac{d-1}{\alpha})}\)}
            \label{plot:complexity_adapt_rate}
        \end{axis}
    \end{tikzpicture}\begin{tikzpicture}
        \begin{axis}[
            %width=0.8\textwidth,
            %height=0.45\textwidth,
            grid = major,
            xlabel={Error $\varepsilon$},
            ylabel={Total Work},
            ymode=log,
            xmode=log,
            legend style={at={(0.99,0.99)},anchor=north east,fill opacity=0.8, draw opacity=1, text opacity=1, fill=white,draw=black},
            x tick label style={/pgf/number format/fixed},
            y tick label style={/pgf/number format/sci}
        ]
        
        \addplot+[
            mark=*,
            black,
            error bars/.cd,
                y dir=both, y explicit,
        ] coordinates {
            (5.42710349e-02, 3.31776000e+05) +- (1.64722030e-05, 1.22070312e-04)
            (1.59917548e-02, 2.74754604e+07) +- (2.66256140e-06, 6.51312312e+03)
            (4.21460107e-03, 1.63391919e+09) +- (4.91055321e-07, 7.74424451e+04)
            (1.04138395e-03, 1.06717674e+11) +- (9.25692434e-08, 3.86366400e+06)
            (2.60842483e-04, 6.69036182e+12) +- (1.77295335e-08, 5.79752012e+07)
            (6.55268647e-05, 4.27547672e+14) +- (3.94362666e-09, 2.82819538e+09)
            (1.63771255e-05, 2.72806540e+16) +- (9.27546960e-10, 2.32197530e+11)
        };
            %\addlegendentry{Adaptive Method}

            \addplot [thick, dashed, draw=blue, domain=5.42710349e-02:1.63771255e-05] {10^(2) * x^(-(2+1.5))};
            %\addlegendentry{\(\varepsilon^{-(\frac{1}{\beta} + \frac{d}{\alpha})}\)}

            \addplot [thick, dashed, draw=green, domain=5.42710349e-02:1.63771255e-05] {10^(2) * x^(-(2+ 1 ))};
            %\addlegendentry{\(\varepsilon^{-(\frac{1}{\beta} + \frac{d-1}{\alpha})}\)}
        \end{axis}
    \end{tikzpicture}
    \caption{Overall algorithm complexity for the $2$-dimensional drop-wave function (left) and the $3$-dimensional Styblinski-Tang function (right) over different values of $L$. These rates match the theoretical rates in \cref{theorem:final_complexity}.}
    \label{fig:styblinski_total_complexity}
\end{figure}

\subsection{Failure Region of a Hyperelastic Beam Modelled by a PDE with Random Field Coefficients} \label{subsection:numerics_failure_region_hyperelastic_beam}

Here, we consider the problem of approximating the failure region of a beam made
of an unknown/uncertain hyperelastic material. The objective here is to
determine the parameter region (beam shape and applied traction force) in which
the beam displacement would exceed a given failure threshold with high
probability. The deterministic PDE problem setup and description has been
adapted from the FEniCSx open-source software library \cite{baratta2023dolfinx}
documentation\footnote{See \url{https://docs.fenicsproject.org/}.}, and from
J.~S.~D\o kken's Dolfinx tutorial \cite{dokken_fenicsx}.

%\textbf{Problem description.} 
We consider a 2-dimensional hyperelastic beam of length $\bar{L}$ and thickness $2w$. Let ${A}=[0,\bar{L}]\times[-w,w]\subset \mathbb{R}^2$ be the 2-dimensional domain that corresponds to the beam at rest. The beam displacement $u: {A} \rightarrow \mathbb{R}^2$ is the vector field that minimizes the total potential energy $\Pi(u)$ of the beam:
\begin{align*}
    \min_{u \in V} \Pi(u) = \min_{u \in V}\left(\int_{{A}} \psi(u) \, \mathrm{d}x
    - \int_{{A}} B \cdot u \, \mathrm{d}x
    - \int_{\partial {A}} T \cdot u \, \mathrm{d}s\right)
\end{align*}
Here $V$ is a suitable function space possibly incorporating Dirichlet boundary conditions,
$\psi(u)$ is the elastic stored energy density, $B$ is a body force (per unit reference volume), and $T$ is a traction force (per unit reference area). 

The beam displacement is found by setting the directional derivative of $\Pi$ to zero for all directions $v\in V$, i.e.,
\begin{align*}
    a(u; v) = D_{v} \Pi = \lim_{\epsilon\rightarrow 0} \frac{d \Pi(u + \epsilon v)}{\epsilon} = 0, \quad \forall v \in V.
\end{align*}
The above identifies a \emph{nonlinear} variational problem: Find $u\in V$ such that
\begin{align}
    \label{eq:hyperelastycity_variational_problem}
    a(u,v) = 0, \quad\forall v \in V.
\end{align}

We define the elastic stored energy density $\psi(u)$ as follows. First, we introduce the deformation gradient
\[
    F(u) = I + \nabla u,
\]
where $I\in\mathbb{R}^{d\times d}$ is the identity matrix.
Then we define the right Cauchy-Green tensor
\[
    C(u) = F^{T}(u) F(u),
\]
and the scalar quantities (called invariants),
\begin{align*}
    J(u) = \det(F(u)), \quad I_{c}(u) = \mathrm{trace}(C(u)).
\end{align*}
We can now introduce the following neo-Hookean model for $\psi$:
\begin{align*}
    \psi(u) =  \frac{\mu}{2} (I_{c}(u) - 3) - \mu \ln(J(u)) + \frac{\lambda}{2}\ln(J(u))^{2},
\end{align*}
where $\lambda$ and $\mu$ are the Lam\'e parameters which are expressed in terms of the Young's modulus $E$ and Poisson ratio $\nu$:
\begin{align*}
    \lambda = \frac{E \nu}{(1 + \nu)(1 - 2\nu)},\quad \mu = \frac{E}{2(1 + \nu)}.
\end{align*}

For our test problem we set $\nu=0.31$ and assume that the material composition of the beam is uncertain and modelled by letting the Young's modulus be a random field, i.e., $E=E(x,\omega)$, where, for all $x$, $E(x,\cdot)$ is a gamma-distributed random variable with mean $\mu_E = 10^7$ and standard deviation $\sigma_E = \frac{3}{4}\times 10^7$ (we consider dimensionless quantities here). We sample realizations of $E$ as follows:
\begin{align*}
    E(x,\omega) = F_\gamma^{-1}(\Phi(z(x,\omega))),
\end{align*}
where $F_\gamma^{-1}(t)$ is the inverse cumulative density function of a gamma
random variable with mean $\mu_E$ and standard deviation $\sigma_E$, $\Phi$ is
the standard Gaussian cumulative density function, and $z(x,\omega)$ is a
zero-mean unit variance Mat\'ern Gaussian field with correlation length
$\bar{\lambda} = 0.1$ and smoothness parameter $\bar{\nu} = 2$. We sample the
Mat\'ern field using the stochastic PDE approach by Croci et
al.~\cite{croci2018efficient,croci2021multilevel}. We remark that due to the
uncertainty in the Lam\'e coefficients, the solution $u$ becomes itself a random
field $u(x,\omega)$.

\textbf{Level-set approximation problem.}
We set $\bar{L}=1$, $r=\bar{L}/w$ (the beam aspect ratio), $B = [0,-9.81]$ (gravity body
force), and $T=[0,\bar{T}]$. We let $\theta = \{r,\bar{T}\}$ be the set of
variable parameters with respect to which we want to approximate the failure
probability. Let $u_f>0$ be a given failure displacement threshold. We define the beam failure probability as
\(
    \mathbb{P}(\lVert u \rVert_{L^{\infty}(A)} > u_f)
\)
wherein both $u$ and $A$ depend on the parameter set $\theta$.

Our objective is to approximate the failure region of the beam, i.e. the
${\theta}$-parameter region in which the beam will not fail with high enough
probability. By setting the failure probability to be $\bar{p}=0.05$ and $u_f = 0.01$, we define the boundary of the failure region as the zero level-set 
\begin{equation*}
        \mathcal{L}_0
    = 
        \left\{
            {\theta} \in [25,75]\times[-100,100]
            :
            f(\theta) := \mathbb{P}(\lVert u \rVert_{L^{\infty}(A)} > u_f) - \bar{p} = 0
        \right\}
\end{equation*}
where the norm within $f$ varies with the set of parameters ${\theta} = (r,\bar{T})$.

\textbf{Implementation details.}
We start from an initial uniform mesh of size $h_0=2^{-3}$. For estimating the
failure probability we use the standard Monte Carlo method which leads to
$\beta=\frac{1}{2}$ and $p=\infty$ (the Monte Carlo statistical error is
asymptotically Gaussian). We set the number of samples on each level to be
$M_\ell=2^{\frac{\alpha}{\beta}(\ell+1)}$. We leverage the capabilities
of the open-source BLUEST Monte Carlo software library from
\cite{croci2023multi}, which employs MPI for parallel Monte Carlo sampling.
Each Monte Carlo sample requires solving a realization of the nonlinear
variational problem \eqref{eq:hyperelastycity_variational_problem}, which we
discretize using the finite element method. In particular, we use piecewise-quadratic
continuous Lagrange finite elements for all variables over a fixed uniform
finite element simplicial mesh of $200$ cells.  We employ the FEniCSx
open-source finite element library \cite{baratta2023dolfinx} and the PETSc
Newton solver \cite{petsc-user-ref} (basic linesearch, relative and absolute
tolerances set to $10^{-8}$, and $100$ maximum iterations). In order to increase
the robustness of the Newton solver, we provide the solution to the
corresponding linear elasticity equations\footnote{Obtained by replacing the
nonlinear elastic stored energy density $\psi$ with the linear density
$$\psi_{\text{lin}}(u) = \mu(\nabla u + \nabla u^T) + \lambda (\nabla \cdot
u)I.$$} as initial guess. 

\textbf{Numerical results.}
The results of the adaptive algorithm with ${L} = 5$ adaptive refinements are shown in \cref{fig:hyperelasticity_level_set_new}, from which one can see heuristically that the target level-set is well approximated, though the approximated level-set is non-smooth due to the discontinuous approximants used. In contrast, \cref{fig:hyperelastic_beam_complexity} shows the work rate of the algorithm over different values of ${L}$, from which we can see that the expected work rate from \cref{lemma:work_bound_full} are attained.
\begin{figure}[!ht]
    %\vspace{-10pt}
    \centering
    \includegraphics[width=0.45\linewidth]{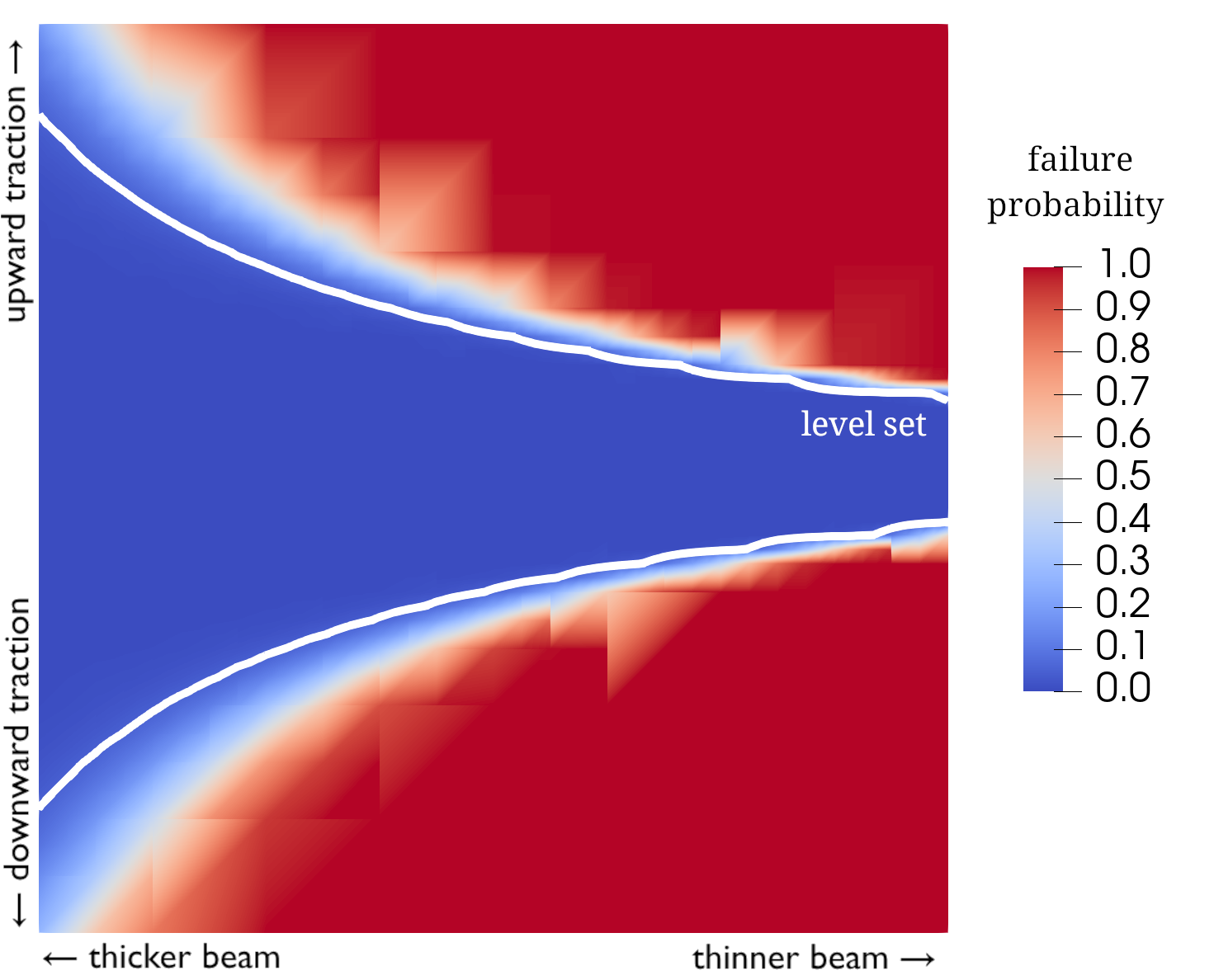}
    \hspace{0.02\linewidth}
    \includegraphics[width=0.36\linewidth]{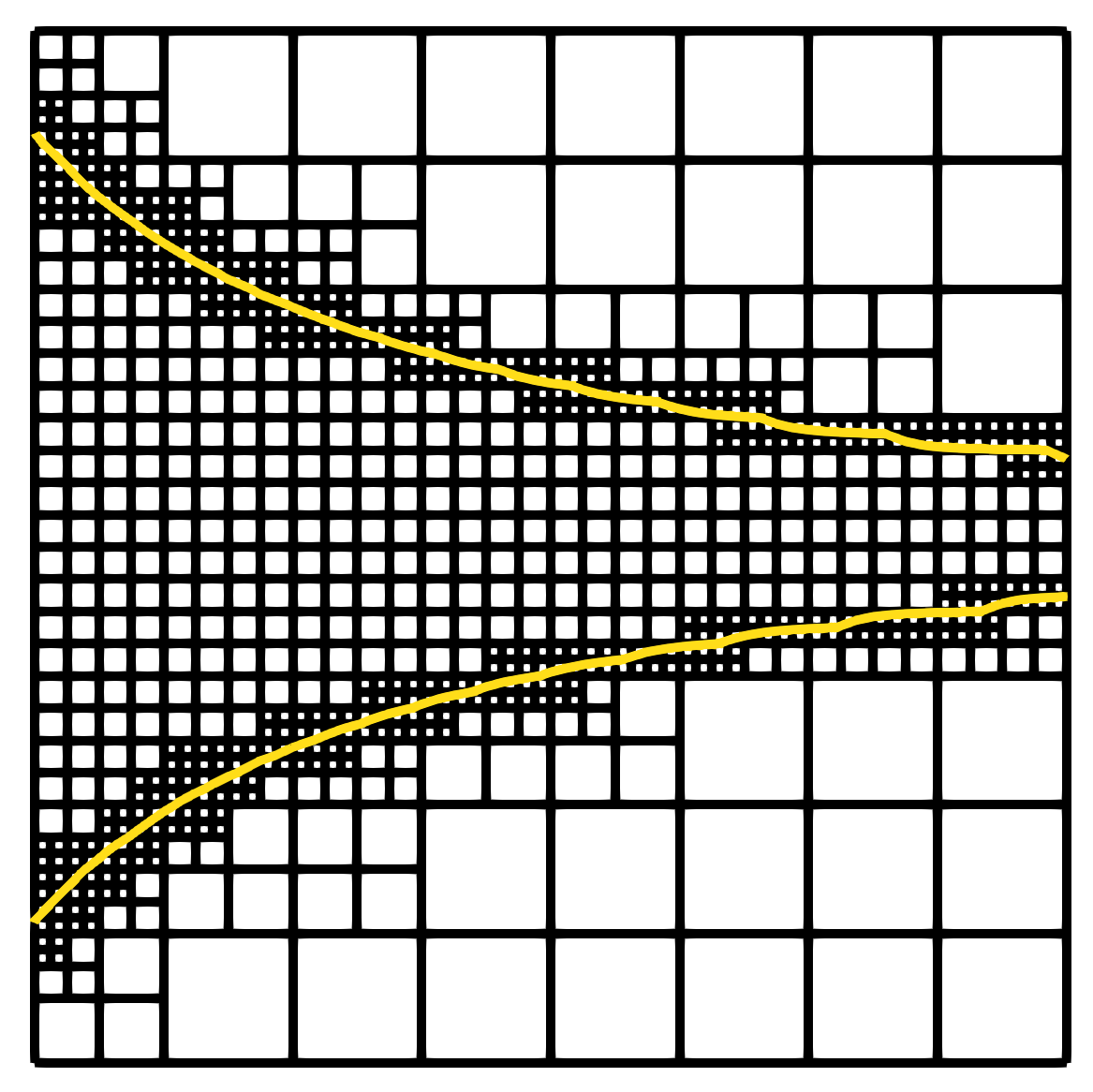}
    %\vspace{-10pt}
    \caption{On the left, the finest approximation obtained of the $0.05$-level set of the hyperelastic beam failure probability. On the right, the corresponding adapted grid. Note that the mesh size is fine close to the level set and coarse away from it, and thus the approximant obtained is consequently only accurate close to the level set.}
    \label{fig:hyperelasticity_level_set_new}
\end{figure}

\begin{figure}[ht]
    \centering
        \begin{tabular}{ccc}
        \ref{plot:hyperelastic_adaptive} Adaptive Method
		& \ref{plot:hyperelastic_full_rate} \(2^{(\alpha/\beta + d) {L}}\) 
        & \ref{plot:hyperelastic_adapt_rate} \(2^{(\alpha/\beta + d-1) {L}}\)
	\end{tabular}
    \begin{tikzpicture}
        \begin{axis}[
            xlabel={${L}$},
            ylabel={Total Work},
            ymode=log,
            grid=major,
            legend style={at={(0.01,0.99)},anchor=north west,fill opacity=0.8, draw opacity=1, text opacity=1, fill=white,draw=black},
            xtick={1,2,3,4},
            x tick label style={/pgf/number format/fixed},
            y tick label style={/pgf/number format/sci}
            ]
            \addplot+[
                mark=*,
                black,
                error bars/.cd,
                    y dir=both, y explicit,
            ] coordinates {
                (1,1.2960000e+03)
                (2,4.7776000e+04)
                (3, 2.2076640e+06)
                (4, 4.3896768e+07)
            };
            %\addlegendentry{Adaptive Method}
            \label{plot:hyperelastic_adaptive}
            
            \addplot [thick, dashed, draw=blue, domain=1:4] {2*10^(2) * 2^((4 + 2)  * x)};
            %\addlegendentry{\(2^{(\alpha/\beta + d) {L}}\)}
            \label{plot:hyperelastic_full_rate}
        
            \addplot [thick, dashed, draw=green, domain=1:4] {4*10^(1) * 2^((4 + 1)  * x))};
            %\addlegendentry{\(2^{(\alpha/\beta + d-1) {L}}\)}
            \label{plot:hyperelastic_adapt_rate}
        \end{axis}
    \end{tikzpicture}
    %%%%%%%%%%%%%%%%%%%%%%%%%%%%%%%%%%%%%%
    \begin{tikzpicture}
        \begin{axis}[
            xlabel={${L}$},
            ylabel={CPU time (s)},
            ymode=log,
            grid=major,
            legend style={at={(0.01,0.99)},anchor=north west,fill opacity=0.8, draw opacity=1, text opacity=1, fill=white,draw=black},
            xtick={1,2,3,4},
            x tick label style={/pgf/number format/fixed},
            y tick label style={/pgf/number format/sci}
            ]
            \addplot+[
                mark=*,
                black,
                error bars/.cd,
                    y dir=both, y explicit,
            ] coordinates {
                (1, 5.2973e+00)
                (2, 3.3702e+01)
                (3, 1.1434273e+03)
                (4, 2.00366994e+04)
            };
            %\addlegendentry{Adaptive Method}
            
            \addplot [thick, dashed, draw=blue, domain=2:4] {2*10^(-2) * 2^((4 + 2)  * x)};
            %\addlegendentry{\(2^{(\alpha/\beta + d) {L}}\)}
        
            \addplot [thick, dashed, draw=green, domain=2:4] {2*10^(-2) * 2^((4 + 1)  * x))};
            %\addlegendentry{\(2^{(\alpha/\beta + d-1) {L}}\)}
        \end{axis}
    \end{tikzpicture}
    \caption{Total algorithm work (left) and CPU time (right) vs. ${L}$ for the  hyperelastic beam problem. These rates reflect the theoretical ones from \cref{lemma:work_bound_full}.}
    \label{fig:hyperelastic_beam_complexity}
\end{figure}
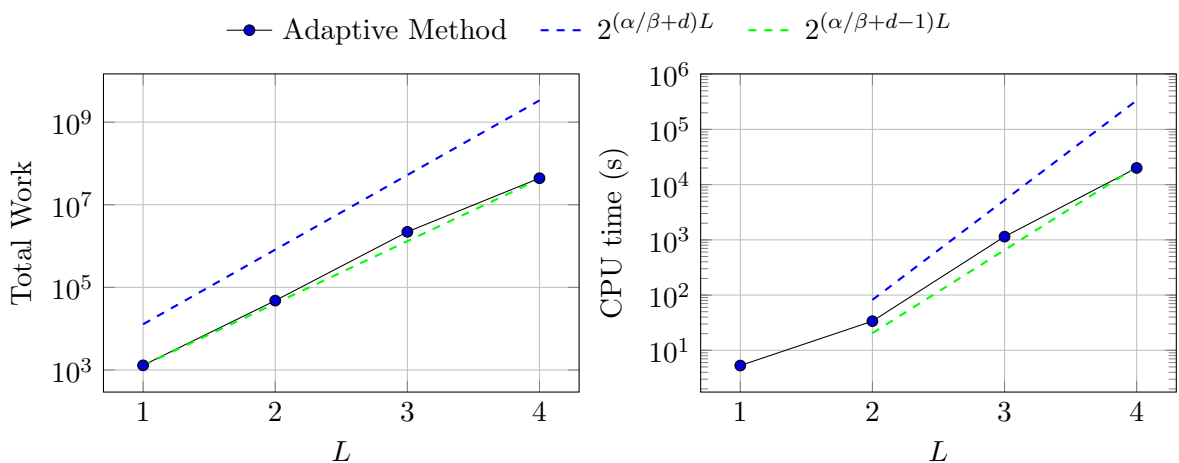

\section{Conclusion} \label{section:conclusion}

In this paper, we have presented an efficient grid-based adaptive sampling algorithm for approximating zero level-sets of possibly noisy functions with unrestrictive assumptions. \cref{algo:main_algo} is robust to noise and is perfectly compatible with advanced Monte Carlo methods such as QMC, MLMC, and MLQMC. This ensures that the presented algorithm can be utilised in a wide variety of applications such as the approximation of failure domains in uncertainty quantification.

The in-depth complexity analysis of the algorithm presented here, backed by numerical experiments, shows that our algorithm attains the same error rate as uniform refinement in $d$ dimensions whilst attaining a more efficient work rate; one that matches our best-case benchmark. \cref{algo:main_algo} is also readily applicable since the refinement criteria can computed with only a priori $L^p$ error rates, and we utilise local approximations so that the algorithm is applicable when global approximation is unavailable. We also provide guidance on how a posteriori error estimates may be utilised with this algorithm (with only minor adjustments) to achieve better constants.

The algorithm presented here ought to be applicable to level-sets of general dimension. Indeed our work analysis, with $d-1$ replaced with the level-set dimensionally (say $d_0$) in the rate, holds if one modifies \cref{assumption:f_int_bound} by changing the RHS to $\rho_0 a^{d-d_0}$. The limitation of the complexity analysis presented here lies in the level-set error metric, which only provides information for level-sets of dimension $d-1$. Hence, a clear avenue for future work is to establish the assumptions required to obtain optimal convergence in more general level-set error metrics.

\section*{Acknowledgements}
The authors would like to acknowledge helpful discussions with R.~D.~Moser and K.~E.~Willcox in the early stages of this project.

M. Croci's work is supported by grant PID2023-146668OA-I00 funded by MICIU / AEI / 10.13039 / 501100011033 and cofunded by the European Union and by grant RYC2022-036312-I funded by MICIU / AEI / 10.13039 / 501100011033 and by ESF+. M. Croci is also supported by the Basque Government through the BERC 2022-2025 program, and by the Ministry of Science and Innovation: BCAM Severo Ochoa accreditation CEX2021-001142-S / MICIN / AEI / 10.13039 / 501100011033.

A.-L. Haji-Ali acknowledges the support of support from the Alexander von Humboldt Foundation through the Experienced Researcher Fellowship.

I. C. J. Powell was supported by the EPSRC Centre for Doctoral Training in Mathematical Modelling, Analysis and Computation (MAC-MIGS) funded by the UK Engineering and Physical Sciences Research Council (grant EP/S023291/1), Heriot-Watt University and the University of Edinburgh.